\documentclass[10pt]{amsart}
\usepackage[T1]{fontenc}            
\usepackage[utf8]{inputenc}         
\usepackage[english]{babel} 
\usepackage{physics}
\usepackage{latexsym}
\usepackage{amsmath}
\usepackage{amssymb}
\usepackage{mathrsfs}
\usepackage{graphicx}
\usepackage{color}
\usepackage{pgfpages}
\usepackage{ifthen}
\usepackage{leftidx,tensor}
\usepackage{mathtools}
\usepackage{comment}
\usepackage{dsfont}
\usepackage[nocompress]{cite}

\usepackage[shortlabels]{enumitem}
\usepackage{aliascnt}
\usepackage[bookmarks=true,pdfstartview=FitH, pdfborder={0 0 0}, colorlinks=true,citecolor=red, linkcolor=blue]{hyperref}
\usepackage{bbm}
\usepackage{nicefrac}

\newtheorem{theorem}{Theorem}[section]

\newtheorem{corollary}[theorem]{Corollary}

\newtheorem{definition}[theorem]{Definition}

\newtheorem{lemma}[theorem]{Lemma}

\newtheorem{proposition}[theorem]{Proposition}
\newtheorem{remark}[theorem]{Remark}

\numberwithin{equation}{section}

\setlist[enumerate]{font = \normalfont}

	\renewcommand{\phi}{\varphi}
	
	\renewcommand{\bar}[1]{\overline{#1}}

\title[Global well-posedness for small data in a 3D coupled model with boundary noise]{Global well-posedness for small data in a 3D temperature-velocity model with Dirichlet boundary noise
  }

  \author{Gianmarco Del Sarto}
\address{Technische Universit\"{a}t Darmstadt\\
Fachbereich Mathematik\\
	Schlossgartenstr.\ 7\\
	64289 Darmstadt\\
	Germany\\
    }
\email{delsarto@mathematik.tu-darmstadt.de}

\author{Marta Lenzi}
\address{Scuola Normale Superiore\\
	Classe di Scienze\\
	P.za dei Cavalieri 7\\
	56126 Pisa\\
	Italy}
\email{marta.lenzi@sns.it}
   \address{University School for Advanced Studies IUSS Pavia \\
Department of Science, Technology and Society\\
P.za della Vittoria 15\\
27100 Pavia\\
Italy
}
\email{marta.lenzi@iusspavia.it}

\begin{document}

\date\today

\begin{abstract}
We study a three-dimensional Boussinesq-type temperature-velocity system on a bounded smooth domain $\mathcal D\subset\mathbb R^3$, where the velocity $u^\varepsilon$ solves the Navier-Stokes equations and the temperature $\theta^\varepsilon$ is driven by Dirichlet boundary noise of intensity $\sqrt{\varepsilon}$. The boundary forcing produces a stochastic convolution $Z^\varepsilon$ which is, in general, only continuous in time with values in $H^{-\frac12-\delta_\theta}(\mathcal D)$. To handle this roughness together with initial data $\theta_0\in W^{s,6/5}(\mathcal D)$, we work in the ambient space $H^{-\frac12-\delta_u}(\mathcal D)$ with
$\delta_u\ge \max\{\delta_\theta,\frac12-s\}$.

Given a finite time $T>0$, for any $p>4$ and sufficiently small initial data, we prove existence and uniqueness of a mild solution $(u^\varepsilon,\theta^\varepsilon)$ up to a stopping time $\tau^\varepsilon\le T$ such that
$$
u^\varepsilon \in W^{1,p}(0,\tau^\varepsilon;H^{-\frac12-\delta_u}(\mathcal D)) \cap
L^p (0,\tau^\varepsilon;H^{\frac32-\delta_u}(\mathcal D)),
\quad
\theta^\varepsilon \in C(0,\tau^\varepsilon;H^{-\frac12-\delta_u}(\mathcal D)).
$$
Moreover, we obtain a high-probability global existence estimate of the form $\mathbb P(\tau^\varepsilon=T)\geq 1- C\varepsilon $, with $C= C( \delta_\theta, T)>0.$
\end{abstract}

\maketitle

\tableofcontents

\section{Introduction}
\label{sec: introduction}

Let $\mathcal{D}\subset\mathbb{R}^3$ be a bounded, open domain (i.e. non-empty, connected set) with smooth boundary $\partial\mathcal{D}$, and let $T>0$ be a fixed final time.  We study the coupled velocity-temperature system for an incompressible fluid in $\mathcal{D}$, modelled by Navier-Stokes equations with thermal advection-diffusion, and perturbed by Dirichlet boundary noise.  Concretely, for each $\varepsilon>0$ we consider
\begin{equation}
\left\{
\begin{aligned}
\partial_t u^\varepsilon + u^\varepsilon \cdot    \nabla u^\varepsilon + \nabla p^\varepsilon  - \Delta u^\varepsilon&=  - \theta^\varepsilon  e_3, \quad &\text{ in } &\mathcal{D} \times (0,T), \\
\mathrm{div}(u^\varepsilon) &= 0, \quad &\text{ in } &\mathcal{D} \times (0,T),\\
u^\varepsilon|_{\partial \mathcal{D} } &= 0, \quad &\text{ in } & \partial\mathcal{D} \times (0,T), \\
\partial_t \theta ^\varepsilon + u^\varepsilon \cdot \nabla \theta^\varepsilon -\Delta \theta ^\varepsilon &= 0, \quad &\text{ in } &\mathcal{D} \times (0,T),\\
\theta_{\mid \partial \mathcal{D} }^\varepsilon &= \sqrt{\varepsilon} \, \frac{dW}{dt}, \quad &\text{ in } & \partial\mathcal{D} \times (0,T), \\
u^\varepsilon|_{t = 0 } &= u_0, \quad \theta^\varepsilon|_{t = 0} =\theta_0. 
\end{aligned}
\right.
\label{eq: full system}
\end{equation}
where $u^\varepsilon\colon\mathcal{D}\times(0,T)\to\mathbb{R}^3$ is the fluid velocity, $\theta^\varepsilon\colon\mathcal{D}\times(0,T)\to\mathbb{R}$ is the temperature field, $p^\varepsilon \colon\mathcal{D}\times(0,T)\to\mathbb{R}$ is the pressure, $e_3=(0,0,1)^T$, and $(W_t)_{t\ge0}$ is a $\mathcal{Q}$-Wiener process acting only on the boundary with intensity which scales depending on the parameter $\varepsilon$.

The parameter $\varepsilon$ influences the probability of global existence up to time $T$.  Indeed, the three-dimensional (3D) Navier-Stokes equations are only known to be globally well‑posed under smallness conditions on initial data and forcing (see, e.g.,~\cite[Chapter 9]{Constantin1988}).  Since stochastic boundary forcing enters the velocity equation through buoyancy coupling, we impose a small‑noise regime in order to retain a nonzero probability of well‑posedness on $[0,T]$.  However, even for arbitrarily small $\varepsilon>0$, the boundary noise may assume large values with positive probability. Consequently, we construct the solution $(u^\varepsilon,\theta^\varepsilon)$ up to a random stopping time $\tau^\varepsilon\leq T$, defined to interrupt the evolution if the stochastic forcing becomes too large.

A key difficulty of the model comes from the fact that Dirichlet boundary noise is
much rougher than the forcing acting in the interior of the domain. Even in the linear heat equation with boundary noise, the stochastic
convolution produced by the boundary forcing typically lives only slightly below
the threshold $H^{-1/2}(\mathcal D)$. We quantify this loss through a small exponent
$\delta_\theta>0$, which measures how far the noise-driven component falls into
$H^{-1/2-\delta_\theta}(\mathcal D)$.

To couple this temperature field back into the three-dimensional Navier-Stokes
equation via the buoyancy term $-\theta^\varepsilon e_3$, we solve the fluid problem
in a (slightly) weaker Sobolev scale $H^{-1/2-\delta_u}(\mathcal D)$, chosen so that
\emph{both} the noise contribution and the initial temperature can be interpreted as
forcing terms at the fluid level. This leads to the compatibility requirement
\[
\delta_u \ge \max\{\delta_\theta,\ \tfrac12-s\},
\]
where $s\in(0,1/2)$ is the Sobolev regularity of the initial temperature
$\theta_0\in W^{s,6/5}(\mathcal D)$.

Under this condition (and for $p>4$ in the maximal-regularity framework), we prove the existence and uniqueness of the coupled system on the interval $[0,\tau^\varepsilon]$, satisfying
\[
u^\varepsilon  \in W^{1,p} (0, \tau^\varepsilon; H^{- \frac{1}{2}-\delta_u} (\mathcal{D})) \cap L^p (0,\tau^\varepsilon; H^{\frac{3}{2}- \delta_u} (\mathcal{D})), \quad \theta^\varepsilon \in C(0, \tau^ \varepsilon; H^{- \frac{1}{2}-\delta_u} (\mathcal{D})).
\]
Furthermore, we establish the high-probability estimate $\mathbb P(\tau^\varepsilon=T)\geq  1-C\varepsilon$. 

It is worth underlining two aspects of the boundary noise: its mathematical difficulties, and its physical meaning.  Firstly, Da Prato and Zabczyk were the first to rigorously demonstrate in \cite{DaPrato1993} that solutions to the heat equation subject to Dirichlet boundary noise exhibit significantly reduced regularity compared to solutions driven by interior stochastic forcing, even in one spatial dimension. Specifically, due to the presence of boundary noise, the best regularity that one can hope for, even in the linear case, is continuity in time with values in a negative-order Bessel potential space, namely $C(0,T;H^{-\frac{1}{2}-\delta_\theta} (\mathcal{D}))$ for any $\delta_\theta > 0$ and any final time $T>0$. This limitation reflects the difficulties of noisy Dirichlet boundary conditions.

The low spatial regularity of the stochastic convolution in the Dirichlet setting is a major obstruction for non-linear problems: even basic products and non-linear
maps may fail to be well-defined in a classical sense, and standard energy methods
or fixed-point arguments often require additional structure. As a consequence, the
available well-posedness theory for non-linear PDEs with Dirichlet white-noise
boundary data is rather limited and typically concerns only a limited class of non-linearities; see for instance \cite{FabbriGoldys,Maslowski,Sowers} and references therein.
Other works treat situations where the boundary forcing is regularised (e.g. coloured in
space and/or fractional in time) or otherwise smoothed at the boundary, which can lead to
better regularity and allow one to handle genuinely non-linear dynamics; see \cite{Blessing,Fan}. In our setting the boundary noise enters the temperature equation linearly, but the resulting
rough temperature acts as a forcing in the Navier-Stokes component, and its interplay with
the convective non-linearity is one of the main analytic difficulties. It is also worth noting that the $H^{-\frac{1}{2}-\delta_\theta }$ limitation is closely related to a
boundary-layer singularity: solutions to the associated linear Dirichlet problem may blow up
as one approaches $\partial\mathcal D$, while remaining smooth in the interior, see \cite{AlosBonaccorsi2002b,AlosBonaccorsi2002a,Brezniak2015,Goldys}.

In parallel, the velocity field $u^\varepsilon$ enjoys, in this context, the maximal regularity of the Stokes operator in the same low-regularity framework introduced above. This constitutes the optimal regularity according to the maximal regularity theory for the Stokes operator on $H^{-\frac{1}{2}-\delta_u}(\mathcal{D})$, with suitable divergence-free and boundary conditions on the spaces. Our maximal regularity approach follows the framework developed by Pr\"uss and Wilke in \cite[Section 5]{Pruss2018}.

Secondly, from a physical point of view, stochastic boundary forcing represents the effect of unresolved fast fluctuations, such as boundary-layer instabilities or small-scale convection, that cannot be described deterministically at the scale of our model. Related approaches that model boundary-layer effects through non-standard boundary mechanisms can be found, for instance, in \cite{BerselliRomito2006}. In climate science, this philosophy traces back to Hasselmann’s stochastic climate paradigm~\cite{Hasselmann1976}, in which slow, large-scale dynamics are driven by fast, random perturbations. The boundary-noise framework thus provides a mathematically tractable way to “close” the system. In this spirit, we plan to investigate in a future work in which sense the system \eqref{eq: full system} arises as a limit of a multiscale fast-slow model.

Lastly, a natural direction for future research is the analysis of the two-dimensional counterpart of the system. In the 2D setting, it is an interesting open problem to determine whether global well-posedness can be established on any time interval without relying on the stopping time $\tau^\varepsilon$ (i.e., independent of the noise intensity).

\section{Preliminaries and main results}
\label{sec: preliminaries and main result}
In this section we introduce the notation which will be used throughout the
paper, describe our approach to solve the coupled velocity-temperature problem \eqref{eq: full system}, and present the main results of our work.
\subsection{Notations and functional setting}
\label{subsec: notations and functional setting}
We work on a complete filtered probability space $\left(\Omega, \mathcal{F}, \left( \mathcal{F}_t \right)_t , \mathbb{P}\right)$. A stochastic process $\Phi$, taking values in a measurable space, is adapted if $\Phi_t$ is $\mathcal{F}_t$-measurable for any $t \geq 0$. It is progressively measurable if the map $(s, \omega) \mapsto \Phi_s(\omega)$ is measurable on $([0,t] \times \Omega, \mathcal{B}(0,t) \otimes \mathcal{F}_t  )$ for every $t \geq 0$, with $\mathcal{B}(0,t)$ being the Borel $\sigma$-algebra on $[0,t].$

Let $\mathcal{D}\subset \mathbb{R}^3$ be a bounded, open domain (i.e. non-empty, connected set) with smooth boundary. We denote by $(W_t)_t$ a $\mathcal{Q}$-Wiener process on $L^2(\partial \mathcal{D})$ defined on $\left(\Omega, \mathcal{F}, \left( \mathcal{F}_t \right)_t , \mathbb{P}\right)$, and represented by
\[
W(t,x) = \sum_{k \geq 0} \lambda_k e_k(x) \beta_k(t), \quad t \geq 0, \ x  \in \partial \mathcal{D}
\]
where $(e_k)_k \subset L^2 (\partial \mathcal{D})$ is an orthonormal basis of $L^2 (\partial \mathcal{D})$, $(\beta_k(t))_k$ are independent Brownian motions, and $(\lambda_k)_k$ are the non-negative square roots of the eigenvalues of the covariance operator $\mathcal{Q}$. For further details, we refer to \cite{DaPrato}. 

To handle the stochastic forcing on the boundary, we employ the \emph{Dirichlet map} $D$, following the approach of \cite{DaPrato1993}. It is defined as the linear map
\[
    D\colon L^2(\partial \mathcal{D}) \to L^2(\mathcal{D}), \quad D h := u,
\]
where $u$ denotes the unique weak solution to the Dirichlet problem 
\begin{equation*}
\left \lbrace
    \begin{aligned}
        \Delta u &= 0,  \quad &x \in &\mathcal{D} ,\\
        u|_{ \partial \mathcal{D}} & = h, \quad  \quad &x \in & \partial \mathcal{D}.
    \end{aligned}
    \right.
\end{equation*}

Regarding the functional setting, for $s \in \mathbb{R}$ and $p \in (1,\infty),$ we denote by $W^{s,p}(\mathcal{D})$ the fractional Sobolev space and by $H^{s,p}(\mathcal{D})$ the Bessel potential space. In particular, we set $H^{0,p}(\mathcal{D})=L^p(\mathcal{D})$, and write $H^{s}(\mathcal{D}) = H^{s,2}(\mathcal{D}) $. Note that $H^{s,p}(\mathcal{D}) = W^{s,p}(\mathcal{D})$ for any $s \in \mathbb{N}.$  

We denote by 
\begin{equation*}
\begin{aligned}
        \Delta \colon H^{2,2}(\mathcal{D}) \cap H^{1,2}_0 (\mathcal{D}) \subset L^2(\mathcal{D}) \to L^2(\mathcal{D})
\end{aligned}
\end{equation*}
the Dirichlet Laplacian, and we consider the fractional powers $(-\Delta)^\alpha$, for any $\alpha \in (0,1)$. Their domains are given by
\begin{equation*}
D((-\Delta) ^\alpha) := \prescript{}{0}{H}^{2\alpha} (\mathcal{D}) := \left \lbrace 
    \begin{aligned}
       & H^{2\alpha}(\mathcal{D}), &\quad &0 < \alpha < \frac{1}{4}, \\
       &H^{\frac{1}{2}}_{00}(\mathcal{D}), & \quad & \alpha = \frac{1}{4},\\
       & H^{2\alpha}_0 (\mathcal{D}), &\quad & \frac{1}{4} < \alpha \leq \frac{1}{2},\\
       & H^{2\alpha }(\mathcal{D}) \cap H^1_0 (\mathcal{D}), & \quad & \frac{1}{2}< \alpha < 1,
    \end{aligned}
    \right.
\end{equation*}
where $H^{\frac{1}{2}}_{00}(\mathcal{D})$ consists of all $u \in H^{\frac{1}{2}}(\mathcal{D})$ such that
\[
\int_{\mathcal{D}} \rho(x)^{-1} \abs{u(x)}^2  dx < \infty, 
\]
with $\rho(x)$ being any $C^\infty $ function comparable to $dist(x, \partial \mathcal{D})$; see \cite{Lions1972, Triebel} for more details on the domain of the fractional powers of the Dirichlet Laplacian.

We define the space of solenoidal (divergence-free) square-integrable vector fields by
\[
L^2_\sigma(\mathcal{D}) = \overline{\left \lbrace u \in C^\infty_c(\mathcal{D}; \mathbb{R}^3) \ : \ \mathrm{div}(u) = 0  \text{ in } \mathcal{D} \right \rbrace}^{\norm{\cdot}_{L^2(\mathcal{D})}}.
\]
The space $L^2(\mathcal{D})$ can be decomposed as
\[
L^2(\mathcal{D}) = L^2_\sigma (\mathcal{D}) \oplus G_2(\mathcal{D}),
\]
where $G_2(\mathcal{D}) := \left \lbrace u \in L^2(\mathcal{D}) \ : \ u = \nabla \pi,  \text{ for some } \pi \in H^{1,2}_{loc}(\mathcal{D}) \right \rbrace $, and there exists a unique projection $P\colon L^2(\mathcal{D}) \to L^2_{\sigma}(\mathcal{D})$ called the \emph{Helmholtz projection} in $L^2(\mathcal{D})$. We denote by $A$ the Stokes operator with Dirichlet boundary conditions, defined by
\[
Au:= - P \Delta u, \quad D(A):= H^{2,2}(\mathcal{D}) \cap H^{1,2}_0 (\mathcal{D}) \cap L^2_{\sigma}(\mathcal{D}).
\]
In the Hilbert space $X_0 = L^2_\sigma(\mathcal{D})$, the Stokes operator $A$
is a non-negative, self-adjoint operator with compact inverse; see, for
instance, \cite[Section~2]{TemamNSEs}. In particular, $A$ is sectorial of angle
$0$ and $-A$ generates a bounded analytic $C_0$-semigroup on $X_0$; see, for
example, \cite{EngelNagel2000,VeraarVolII}. Moreover, since $A$ is a
non-negative self-adjoint operator on a Hilbert space, it admits a bounded
$\mathcal{H}^\infty$-functional calculus of angle $0$; see
\cite[Chapters~2 and~7]{HaaseFunctionalCalculus}. In the specific case of the
Stokes operator, the domains of the fractional powers $A^\alpha$ in
$L^2_\sigma(\mathcal{D})$ have been identified in \cite{Giga1985} and are given
by
\[
D(A^\alpha)
  = \prescript{}{0}{H^{2\alpha}}(\mathcal{D}) \cap L^2_\sigma(\mathcal{D}),
  \qquad \alpha \in (0,1).
\]
For notational convenience, we introduce the following spaces:
\[
H^s_\sigma(\mathcal{D}) = H^s(\mathcal{D}) \cap L_{\sigma}^2(\mathcal{D}), \quad \prescript{}{0}{H}^s_\sigma(\mathcal{D})  = \prescript{}{0}{H^s(\mathcal{D})} \cap L^2_\sigma (\mathcal{D}), \quad s \geq 0,
\]
and, by duality
\[
H^{-s}_\sigma(\mathcal{D}) =  \left(H^{s}(\mathcal{D}) \cap L^2_\sigma (\mathcal{D}) \right)', \quad \prescript{}{0}{H}^{-s}_\sigma(\mathcal{D})  =  \left(\prescript{}{0}{H^{s}(\mathcal{D})} \cap L^2_\sigma (\mathcal{D})\right)', \quad s >0.
\]

We will work with an extension of the classical Stokes operator to a weaker setting, which we denote by $A_w$. Specifically, for any sufficiently small $\delta_u>0$, we define the \emph{weak Stokes operator} by
\begin{equation}
    \begin{aligned}
        A_w \colon \prescript{}{0}{H}^{\frac{3}{2}-\delta_u}_\sigma (\mathcal{D}) &\to \prescript{}{0}{H^{- \frac{1}{2}-\delta_u}_\sigma  (\mathcal{D})} \\
        \langle  A_w u, v \rangle &:=  \langle \nabla u, \nabla v \rangle_{{H^{\frac{1}{2}-\delta_u } }, {H^{- \frac{1}{2}+ \delta_u} }}, 
    \end{aligned}
    \label{eq: definition weak stokes operator}
\end{equation}
for all pairs $(u,v) \in \prescript{}{0}{H}^{\frac{3}{2}-\delta_u}_\sigma (\mathcal{D}) \times  \prescript{}{0}{H^{ \frac{1}{2}+\delta_u}_\sigma  (\mathcal{D})} $. Here, $\langle \cdot, \cdot \rangle_{{H^{\frac{1}{2}-\delta_u } }, {H^{- \frac{1}{2}+ \delta_u} }}$ denotes the duality pairing between ${H^{\frac{1}{2}-\delta_u} (\mathcal{D})}$ and ${H^{-\frac{1}{2} + \delta_u} (\mathcal{D})}$.

Lastly, given a final time $T>0$ and $p \in (1,\infty)$, we set
\[
\mathbb{E}_{t,p}^{\delta_u} := W^{1,p}(0,t; \prescript{}{0}H^{-\frac{1}{2}-\delta_u }_\sigma (\mathcal{D})) \cap L^p(0,t; \prescript{}{0} H^{\frac{3}{2}-\delta_u}_\sigma  (\mathcal{D})),
\]
for any $t \in [0,T]$. This space will be the maximal regularity space on which the velocity of our model evolves. Additionally, we define the real interpolation space
\[
V_{ p}^{\delta_u} = \left( \prescript{}{0}H^{-\frac{1}{2}-\delta_u }_\sigma (\mathcal{D}), \prescript{}{0} H^{\frac{3}{2}-\delta_u}_\sigma  (\mathcal{D}) \right)_{1-\frac{1}{p}, p}.
\]
For more information on these function spaces, we refer the reader to \cite{Lions1972, Triebel,Lunardi2018}.

\subsection{Description of our approach}
The approach we adopt to study model \eqref{eq: full system} is the following. First, we decouple the model \eqref{eq: full system} into a system for temperature only, given a prescribed fluid velocity; and a system for velocity only, given a prescribed temperature. Namely, we study first the equation for the temperature $\theta^\varepsilon \colon \mathcal{D} \times (0,T) \to \mathbb{R}$, which reads as
\begin{equation}
\left\{
\begin{aligned}
\partial_t  \theta^\varepsilon + u \cdot \nabla \theta^\varepsilon  -\Delta \theta ^\varepsilon &= 0, \quad &\text{ in } &\mathcal{D} \times (0,T),\\
\theta ^\varepsilon|_{ \partial \mathcal{D} } &= \sqrt{\varepsilon}  \frac{dW}{dt}, \quad &\text{ in } & \partial\mathcal{D} \times (0,T), \\
\theta^\varepsilon|_{t = 0} &=\theta_0, \quad &\text{ in } & \mathcal{D}.
\end{aligned}
\right.
\label{eq: temperature for given velocity}
\end{equation}
Here, $u \colon \mathcal{D} \times (0,T) \to \mathbb{R}^3$ is a given velocity and $\theta_0$ is the initial condition. Second, we study the 3D Navier-Stokes equations for the velocity $u \colon \mathcal{D} \times (0,T) \to \mathbb{R}^3$ given by
\begin{equation}
\left\{
\begin{aligned}
\partial_t u + u \cdot    \nabla u + \nabla p  - \Delta u&=   f^\theta, \quad &\text{ in } &\mathcal{D} \times (0,T), \\
\mathrm{div}(u) &= 0, \quad &\text{ in } &\mathcal{D} \times (0,T),\\
u|_{\partial \mathcal{D} } &= 0, \quad &\text{ in } & \partial\mathcal{D} \times (0,T), \\
u|_{t = 0 } &= u_0, \quad &\text{ in } & \mathcal{D}.
\end{aligned}
\right.
\label{eq: 3D NSEs velocity only}
\end{equation}
Here, $f^\theta \colon \mathcal{D} \times (0,T) \to \mathbb{R}^3$ is a prescribed forcing, $p\colon \mathcal{D} \times (0,T) \to \mathbb{R}$ is the pressure, and $u_0$ denotes the initial condition.

Then, following the idea in \cite{DaPrato02,Luongo2024}, we split the analysis of \eqref{eq: temperature for given velocity} into two problems, corresponding to the decomposition $\theta_t^\varepsilon = Z_t^\varepsilon + \zeta_t^\varepsilon.$ The first term $Z^\varepsilon_t$ corresponds to the stochastic linear problem with non-homogeneous Dirichlet boundary conditions given by 
\begin{equation}
\left\{
\begin{aligned}
\partial_t Z ^\varepsilon & = \Delta Z ^\varepsilon, \quad &\text{ in } & \mathcal{D} \times(0,T), \\
Z^\varepsilon|_{\partial \mathcal{D}} & = \sqrt{\varepsilon} \frac{dW}{dt}, \quad &\text{ in } & \partial \mathcal{D} \times(0,T), \\
Z^\varepsilon|_{t = 0} & = 0,  \quad &\text{ in } &  \mathcal{D}.
\end{aligned}
\right.
\label{eq: main results stochastic heat with noise on the boundary}
\end{equation}
The solution of the previous equation can be interpreted in mild form, as in \cite{DaPrato1993,Goldys}. Next, we consider the remainder $\zeta_t^\varepsilon:= \theta_t^\varepsilon - Z_t^\varepsilon$, which satisfies
\begin{equation}
\left\{
\begin{aligned}
\partial_t \zeta^\varepsilon + u \cdot \nabla \zeta^\varepsilon + u \cdot \nabla Z^\varepsilon - \Delta \zeta ^\varepsilon  & = 0, \quad &\text{ in } & \mathcal{D} \times(0,T), \\
\zeta ^\varepsilon|_{\partial \mathcal{D}} & = 0, \quad &\text{ in } & \partial \mathcal{D} \times(0,T), \\
\zeta ^\varepsilon|_{t = 0} & = \theta_0,  \quad &\text{ in } &  \mathcal{D}.
\end{aligned}
\right.
\label{eq: zeta temperature only}
\end{equation}
Here, we interpret again the solution $\zeta^\varepsilon \colon \mathcal{D} \times (0,T) \to \mathbb{R}$ in the mild form, namely
\[
\zeta_{t}^\varepsilon=e^{t\Delta }\theta _{0}-\int_{0}^{t}e^{\left(  t-r\right)  \Delta}\left(
u_{r}\cdot\nabla\zeta_{r}^\varepsilon\right)  dr-\int_{0}^{t}e^{\left(  t-r\right)
\Delta}\left(  u_{r}\cdot\nabla Z_{r}^\varepsilon\right)  dr.
\]
We conclude this section by describing the solution of the linear temperature problem associated to \eqref{eq: temperature for given velocity}, i.e. by recalling the following result concerning the mild solution of \eqref{eq: main results stochastic heat with noise on the boundary}, which follows from \cite{DaPrato1993} and the factorisation trick for the stochastic convolution \cite[Section 5.3.1]{DaPrato}.
\begin{proposition}
\label{prop: regularity for Z_t}
For any $\delta_\theta  >0$, fix 
\[
    \alpha_\theta  = \frac{1}{4}+ \frac{\delta_\theta }{2}, \quad \beta_\theta  = \frac{1}{4}- \frac{\delta_\theta}{4}.
\]
Assume that
\begin{equation}
    \sum_{k} \lambda_k^2 \norm{(-\Delta)^{\beta_\theta} De_k}_2^2 < \infty.
    \label{eq: convergence series eigenvalue noise for continuity xi}
    \end{equation}
    The following holds.
    \begin{enumerate}
        \item[(i)]  The process $Z^\varepsilon_t = - \sqrt{\varepsilon} \int_0^t \Delta e^{(t-r)\Delta} D dW_r$ is the unique $\left(\mathcal{F}_t\right)_t$-adapted mild solution of \eqref{eq: main results stochastic heat with noise on the boundary}, and takes values in $H^{- 2\alpha_\theta }(\mathcal{D})$.
        \item[(ii)] The process
    \begin{equation}
    \xi^\varepsilon_t := (-\Delta) ^{-\alpha_\theta } Z^\varepsilon_t =  \sqrt{\varepsilon} \int_0^t (-\Delta)^{1- \alpha_\theta } e^{(t-r)\Delta}D dW_r,
        \label{eq: definition of xi_t}
    \end{equation}
    has a continuous version in $L^2 (\mathcal{D})$ and $Z^\varepsilon$ has a continuous version in $H^{-2\alpha_\theta }(\mathcal{D})$.
    \item[(iii)] Let $T>0$. For any $\vartheta>0$, there exists a positive constant $C = C(\delta_\theta , T)$ such that
    \[
    \mathbb{P}\left( \sup_{0 \leq t \leq T} \norm{Z_t^\varepsilon}_{H^{-2\alpha_\theta}(\mathcal{D})}> \vartheta  \right) \leq \frac{\varepsilon}{\vartheta^2}  C \sum_k \lambda_k^2 \norm{(-\Delta)^{\beta_\theta  }D e_k}_2^2.
    \]
    \end{enumerate}
\end{proposition}
We remark that point (iii) of the previous result is key since it controls, depending on the noise intensity $\sqrt{\varepsilon}$, the probability of the sup norm of $Z_t^\varepsilon$ to be small in an appropriate Bessel potential space.

\subsection{Main results}
We are now in a position to state the main results of our work. For this purpose, we begin by introducing some notation.

Let $p>4$ and $s\in(0,\tfrac12)$. Choose
\begin{equation}
\gamma:=\frac{1}{4} + \varepsilon_\gamma, 
\qquad 
0<\varepsilon_\gamma<\min \left\lbrace 
\frac{1-2s}{4},\ \frac{(1-2s)(p-4)}{4p}
\right \rbrace,
\label{eq:param-gamma}
\end{equation}
and set $a(s,\gamma):=\frac{1}{2}+\gamma+\frac{s}{2}$ and
$\lambda_{\max}:=p(1-a(s,\gamma))$. Then
\[
a(s,\gamma)=\frac{3}{4}+\frac{s}{2}+\varepsilon_\gamma<1
\quad\text{(since }\varepsilon_\gamma<\tfrac{1-2s}{4}\text{)},
\]
and, using $1-a(s,\gamma)=\frac{1-2s}{4}-\varepsilon_\gamma$,
\begin{equation}
\lambda_{\max}-\left(\frac{1}{2}-s\right)
=p\left(\frac{1-2s}{4}-\varepsilon_\gamma \right)  -  \frac{1-2s}{2}
= \frac{(1-2s)(p-2)}{4}- p \varepsilon_\gamma
>0,
\label{eq:lmax-lmin}
\end{equation}
so $\lambda_{\max}>\frac12-s$. Moreover,
\begin{equation}
\lambda_{\max}-(1-2s)
= p\left(\frac{1-2s}{4}-\varepsilon_\gamma \right)-(1-2s)
= \frac{(1-2s)(p-4)}{4}-p\varepsilon_\gamma>0,
\label{eq: lambda max minus 12s positive}
\end{equation}
by \eqref{eq:param-gamma}. Hence the interval
$(0,\min\{2s,\ \lambda_{\max}-(1-2s)\})$ is non-empty, and we may choose
\begin{equation}
\lambda:=\left(\frac{1}{2}-s\right)+\varepsilon_\lambda, \qquad
0<\varepsilon_\lambda<\min\left \lbrace2s, \lambda_{\max}-(1-2s)\right \rbrace.
\label{eq:param-lambda}
\end{equation}
In particular, $\lambda<\frac12+s<1$ (since $\varepsilon_\lambda<2s$).
Finally, define the thresholds
\begin{equation}
\bar{\delta}_\theta:=\min\left \lbrace \lambda, 2\gamma-\frac{1}{2}\right \rbrace, \qquad \bar{\delta}_u:=\min\left \lbrace \lambda_{\max}-\lambda, 1-\lambda\right \rbrace .
    \label{eq:param-bar-deltas}
\end{equation}

The first main result concerns the well-posedness of the system \eqref{eq: zeta temperature only}
\begin{theorem}
Let $T>0,\, s\in(0,\tfrac{1}{2})$ and $p>4$, and choose $\gamma,\lambda$ as in \eqref{eq:param-gamma} and \eqref{eq:param-lambda}. Let $\bar\delta_\theta(s,p)$ and
$\bar\delta_u(s,p)$ be given by \eqref{eq:param-bar-deltas}. Then for every
\[
0<\delta_\theta < \bar\delta_\theta(s,p),
\qquad
0<\delta_u<\bar\delta_u(s,p),
\]
the following holds. Set
\[
\alpha_\theta=\frac{1}{4}+\frac{\delta_\theta}{2},\qquad
\beta_\theta=\frac{1}{4}-\frac{\delta_\theta}{4},
\]
and let $Z_t^\varepsilon=-\sqrt{\varepsilon}\int_0^t \Delta e^{(t-r)\Delta}D\,dW_r$
be the continuous version in $H^{-2\alpha_\theta}(\mathcal D)$ from
Proposition \ref{prop: regularity for Z_t}. Then, for every divergence-free velocity field
\[
u \in L^p\bigl(0,T;\prescript{}{0}H^{\frac32-\delta_u}_\sigma(\mathcal D)\bigr)
\cap L^\infty\bigl(0,T;\prescript{}{0}H^{\frac12-\delta_u}_\sigma(\mathcal D)\bigr)
\]
and every initial datum $\theta_0\in W^{s,6/5}(\mathcal D)$, there exists a unique mild
solution $\zeta^\varepsilon$ of \eqref{eq: zeta temperature only} of the form
\[
\zeta_t^\varepsilon = e^{t\Delta}\theta_0
  -\int_0^t e^{(t-r)\Delta}\bigl(u_r\cdot\nabla\zeta_r^\varepsilon\bigr)\,dr
  -\int_0^t e^{(t-r)\Delta}\bigl(u_r\cdot\nabla Z_r^\varepsilon\bigr)\,dr,
\]
with trajectories in
\[
\zeta^\varepsilon\in C(0,T;W^{s,6/5}(\mathcal D))
\qquad \mathbb{P}\text{-a.s.}
\]
\label{thm: existence mild solution zeta}
\end{theorem}

A straightforward but important consequence of the proof of the previous result is the following estimate.
\begin{corollary}
Under the assumptions of Theorem \ref{thm: existence mild solution zeta}, let $\zeta \in C(0,T; W^{s, \frac{6}{5}}(\mathcal{D}))$ and define
    \[
    \tilde{\zeta}_t:= e^{t\Delta}\theta_{0}-\int_{0}^{t}e^{\left(  t-r\right)  \Delta}\left(u_{r}\cdot\nabla\zeta_{r}\right)  dr- \int_{0}^{t}e^{\left(  t-r\right)  \Delta}\left(  u_{r}%
\cdot\nabla Z_{r}^\varepsilon\right)  dr,
    \]
    Then, for any $0 < \lambda < 1-\delta_u$, it holds
    \begin{equation}
        \begin{split}
            \norm{\tilde{\zeta}}_{C(0,T; W^{s, \frac{6}{5}}(\mathcal{D}))} &\leq  \norm{e^{t\Delta }\theta_0}_{C(0,T; W^{s, \frac{6}{5}}(\mathcal{D}))} \\
            &\quad + C \norm{u}_{L^{\frac{p}{\lambda + \delta_u }}(0,T; W^{\frac{1}{2}+\lambda,2}(\mathcal{D}))}\left( \norm{\zeta}_{C(0,T; W^{s, \frac{6}{5}}(\mathcal{D}))} + \norm{\xi^\varepsilon}_{C(0,T; L^2(\mathcal{D}))}\right),
        \end{split}
        \label{eq: estimate tilde tau}
    \end{equation}
    where $C = C(p,\lambda, \delta_u,\delta_\theta, s,T)>0$.
    \label{cor: estimate tilde zeta}
\end{corollary}
So far, in Proposition \ref{prop: regularity for Z_t}, Theorem \ref{thm: existence mild solution zeta}, and Corollary \ref{cor: estimate tilde zeta}, we have investigated the well-posedness of the temperature equation \eqref{eq: temperature for given velocity}. Moreover, we can control the norm of its solution $\theta^\varepsilon = Z^\varepsilon + \zeta^\varepsilon$ with high probability depending on $\varepsilon$, since we are able to estimate the norms of $Z^\varepsilon$ (by Proposition \ref{prop: regularity for Z_t}) and $\zeta^\varepsilon$ (by Corollary \ref{cor: estimate tilde zeta}).

We now turn to the analysis of the 3D Navier-Stokes equations \eqref{eq: 3D NSEs velocity only}, assuming that $f^\theta$ is a prescribed forcing term. Global well-posedness results for the 3D Navier--Stokes equations are well known under smallness assumptions on the initial datum and the forcing; see for instance \cite[Chapter 9]{Constantin1988}. In the classical energy setting one typically assumes $f^\theta\in L^2(0,T;L^2(\mathcal D))$ and $u_0\in H^1(\mathcal D)$.

In our coupled model the forcing is of buoyancy type, $f^\theta=-\theta^\varepsilon e_3$, and the temperature is decomposed as
$\theta^\varepsilon=Z^\varepsilon+\zeta^\varepsilon$. The noise component $Z^\varepsilon$ is, in general, only continuous in time with values in
$H^{-\frac12-\delta_\theta}(\mathcal D)$, for any $\delta_\theta>0$ compatible with the covariance summability condition \eqref{eq: convergence series eigenvalue noise for continuity xi}; the remainder $\zeta^\varepsilon$ is controlled in $W^{s,6/5}(\mathcal D)$. We therefore work in the ambient space
$H^{-\frac12-\delta_u}(\mathcal D)$ for the Stokes/Navier-Stokes maximal regularity theory, choosing $\delta_u$ so that 
\[
\delta_u \geq \max\{\delta_\theta,\ \tfrac12-s\},
\]
which ensures that both $Z^\varepsilon$ and $\zeta^\varepsilon$ can be interpreted as forcing terms in $H^{-\frac12-\delta_u}(\mathcal D)$. Accordingly, we formulate the Navier-Stokes well-posedness result for general forcing
\[
f^\theta \in L^p (0,T;H^{-\frac12-\delta_u}(\mathcal D) ), \quad  p>\frac{2}{1-\delta_u},
\]
which is the natural class arising from the maximal regularity framework in the low-regularity setting.

The notion of solution considered for the 3D Navier-Stokes problem \eqref{eq: 3D NSEs velocity only} is the one of mild solution in the weak setting determined by the weak Stokes operator $A_w$. We now formalise this notion.
\begin{definition}
    A function $u$ is a solution of the 3D Navier-Stokes equations \eqref{eq: 3D NSEs velocity only} on the time interval $[0,T]$ if
    \begin{equation}
    u_t = e^{-t A_w } u_0 +\int_0^t e^{-(t-r) A_w} P\left(  -u_r \cdot \nabla u_r + f_r^\theta \right) dr, 
    \label{eq: variation of constants formula}
    \end{equation}
    for any $0 \leq t \leq T$, and possesses the regularity
    \[
    u \in \mathbb{E}_{T,p}^{\delta_u} = W^{1,p}(0,T;\prescript{}{0}H^{-\frac{1}{2}-\delta_u }_\sigma (\mathcal{D})) \cap L^p(0,T; \prescript{}{0}H^{\frac{3}{2}-\delta_u }_\sigma (\mathcal{D})).
    \]
    \label{def: mild solution weak sense 3D NSEs}
\end{definition}
Here, and throughout the rest of the paper, $P$ denotes the Helmholtz projection, defined by interpolation from $H^{-\frac{1}{2}- \delta_u} (\mathcal{D})$ to $H^{-\frac{1}{2}- \delta_u}_\sigma  (\mathcal{D})$. The global existence results for the velocity problem read as follows.
\begin{theorem}[Global well-posedness for small data]
    \label{thm: global for small data NSEs with force}
Let $\delta_u \in (0,1)$, $f^\theta  \in L^p(0,T; H^{-\frac{1}{2}-\delta_u} (\mathcal{D}))$, with
$ p > \frac{2}{1-\delta_u}$, and $u_0 \in V_{p}^{\delta_u}$. There exist $\tilde \eta =\tilde  \eta(p, \delta_u, T)>0$ and $M = M(p,\delta_u,T)>0$ such that if $\eta \in (0, \tilde{\eta})$ and
\[
\max \left( \norm{u_0}_{V_{ p}^{\delta_u}}, \norm{f^\theta }_{L^p(0,T; H^{-\frac{1}{2}-\delta_u}(\mathcal{D}))} \right) \leq \frac{\eta}{4 M} , 
\]
then there exists a unique solution $u \in \mathbb{E}_{T,p}^{\delta_u}$ with $\| u \|_{\mathbb{E}_{T,p}^{\delta_u}} \leq \eta$ of the 3D Navier-Stokes equation \eqref{eq: 3D NSEs velocity only} in the sense of Definition \ref{def: mild solution weak sense 3D NSEs}.
\end{theorem}
\begin{remark}
Note that Definition \ref{def: mild solution weak sense 3D NSEs} is well-posed due to the fact that, as will be demonstrated in Section \ref{subsec: stokes operator in weak asetting and max reg}, the operator $A_w$ admits a bounded $\mathcal{H}^\infty$-calculus on $\prescript{}{0}H^{-\frac{1}{2}-\delta_u}_\sigma (\mathcal{D})$. Thus, the operator $-A_w$ generates an analytic semigroup on $\prescript{}{0}H^{-\frac{1}{2}-\delta_u}_\sigma (\mathcal{D})$, ensuring that the variation of constants formula \eqref{eq: variation of constants formula} is meaningful, thanks also to the estimate of the convective term
\[
\norm{P(u \cdot  \nabla u)}_{L^p(0,T; \prescript{}{0}H^{-\frac{1}{2}-\delta_u}_\sigma (\mathcal{D}))} \lesssim  \norm{u}^2_{\mathbb{E}_{T,p}^{\delta_u}},
\]
which will be proved in Lemma \ref{lemma: estimate convective term}.
\end{remark}
We now shift our focus to the fully coupled velocity-temperature system \eqref{eq: full system}. The corresponding notion of solution is introduced below.

\begin{definition}
    \label{def: notion of solution}
   The triple $(u^{\varepsilon }, \theta^{\varepsilon }, \tau^\varepsilon)$ is a \emph{solution} of \eqref{eq: full system} on $[0,T]$ if $\tau^\varepsilon \colon \Omega \to [0,T]$ is a stopping time and $(u^{\varepsilon }, \theta^{\varepsilon })$ is a stochastic process with trajectories $\mathbb{P}$-a.s. in 
\[
\mathbb{E}^{\delta_u}_{\tau^\varepsilon ,p}
      \times C\bigl(0,\tau^\varepsilon;H^{- \tfrac{1}{2}-\delta_u}(\mathcal{D})\bigr),
   \]
   where $u^{\varepsilon}$ is adapted to $(\mathcal{F}_t)_t$ as a process in
   $\prescript{}{0}H^{- \frac{1}{2}-\delta_u }_\sigma (\mathcal{D})$, and progressively
   measurable as a process in
   $\prescript{}{0}H^{ \frac{3}{2}-\delta_u }_\sigma (\mathcal{D})$, while
   $\theta^{\varepsilon}$ is adapted as a process in
   $H^{- \frac{1}{2}-\delta_u}(\mathcal{D})$, and, for any $0\leq t \leq \tau^\varepsilon$, it holds 
     \begin{equation*}
         \begin{aligned}
             u_t^{\varepsilon} &= e^{-tA_w} u_0 + \int_0^t e^{-(t-r)A_w} P\left(  -u_r^{\varepsilon} \cdot \nabla u_r^{\varepsilon } - \theta ^{\varepsilon}_r  e_3 \right)dr,\\
             \theta_t^{\varepsilon} &= \zeta_t^\varepsilon + Z_{t}^\varepsilon, \\
             \zeta_t^{\varepsilon} &= e^{t\Delta }\theta_{0}-\int_{0}^{t}e^{\left(  t-r\right)  \Delta}\left(
u_{r}^{\varepsilon}\cdot\nabla\zeta_{r}^{\varepsilon }\right)  dr-\int_{0}^{t}e^{\left(  t-r\right)
\Delta}\left(  u_{r}^{\varepsilon}\cdot\nabla Z_{r}^\varepsilon\right)  dr,\\
             Z_t^\varepsilon &= - \sqrt{\varepsilon} \int_0^t \Delta e^{(t-r)\Delta }  D dW_r,
         \end{aligned}
     \end{equation*}
    with probability one.
\end{definition}
Lastly, we state the main result of our work.

\begin{theorem}
\label{thm: main theorem}
Let $T>0$, $s\in(0,\tfrac{1}{2})$ and $p>4$, and choose $\gamma,\lambda$ as in \eqref{eq:param-gamma} and \eqref{eq:param-lambda}. Let $\bar\delta_\theta$ and $\bar\delta_u$ be given by \eqref{eq:param-bar-deltas}. Let parameters $\delta_\theta, \delta_u$ satisfy
\[
0 < \delta_\theta <  \bar\delta_\theta,
\qquad 0 < \delta_u< \bar\delta_u,
\]
subject to the compatibility conditions
\begin{equation}
\delta_u \geq \max\left\{\delta_\theta, \frac{1}{2}-s\right\},  \qquad
\delta_u < 1 - \frac{2}{p},
    \label{eq: compatability conditions}
\end{equation}
which define a non-empty set for the given range of $s$ and $p$. Set
\[
\alpha_\theta:= \frac1 4+   \frac{\delta_\theta}{2},
\qquad \beta_\theta:=\frac1 4-\frac{\delta _\theta}{4 },
\]
and assume
\[
\sum_k \lambda_k^2 \|(-\Delta)^{\beta_\theta}De_k\|_{L^2(\mathcal D)}^2<\infty.
\]
Then there exist constants $\eta>0$ and
$\widetilde M=\widetilde M(\delta_\theta,\delta_u,T)\ge 2$ such that the following holds:
if
\[
(\theta_0,u_0)\in W^{s,6/5}(\mathcal D)\times V_p^{\delta_u} ,
    \qquad
\max\{\|\theta_0\|_{W^{s,6/5}(\mathcal D)},\ \|u_0\|_{V_p^{\delta_u}}\}
\leq \frac{\eta}{16\widetilde M},
\]
then for every $\varepsilon>0$ there exists a unique solution
$(u^\varepsilon,\theta^\varepsilon,\tau^\varepsilon)$ of \eqref{eq: full system}
on $[0,T]$ in the sense of Definition~\ref{def: notion of solution}, such that
\[
\|u^\varepsilon\|_{\mathbb E^{\delta_u}_{\tau^\varepsilon,p}}  \leq \eta,
\qquad
\|\theta^\varepsilon\|_{C\bigl(0,\tau^\varepsilon;H^{-\frac12-\delta_u}(\mathcal D)\bigr)}\le \eta,
\qquad \mathbb{P}-\text{a.s.}
\]
Moreover, there exists a constant $C(\delta_\theta,T)>0$ such that
\[
\mathbb{ P}(\tau^\varepsilon=T ) \geq 1-\frac{64\,\widetilde M^2\,\varepsilon}{\eta^2}\,
C(\delta_\theta,T) 
\sum_k  \lambda_k^2 \|(-\Delta)^{\beta_\theta}De_k\|_{2}^2.
\]
\end{theorem}

\subsection{Overview}
The rest of this work is organised as follows. In Section~\ref{sec: temperature problem}, we present the analysis of the temperature problem~\eqref{eq: temperature for given velocity}, assuming a given velocity field with specified regularity. The strategy is to split the problem into a linear system with non-homogeneous Dirichlet boundary noise~\eqref{eq: main results stochastic heat with noise on the boundary}, and the remainder system~\eqref{eq: zeta temperature only}. In Section~\ref{sec: velocity problem}, we analyse the velocity system~\eqref{eq: 3D NSEs velocity only}, subject to a prescribed temperature forcing of specified regularity. In particular, we show that the weak Stokes operator admits a bounded $\mathcal{H}^\infty$-calculus and enjoys maximal regularity. We then apply these results to prove a global well-posedness result for small data for the 3D Navier–Stokes equations~\eqref{eq: 3D NSEs velocity only}. Lastly, Section~\ref{sec: well posedness coupled model} is entirely devoted to the proof of Theorem~\ref{thm: main theorem}.

\section{Temperature problem}
\label{sec: temperature problem}
In this section we describe how to solve the temperature problem \eqref{eq: temperature for given velocity}, which is split in the linear problem \eqref{eq: main results stochastic heat with noise on the boundary}, and the remainder problem \eqref{eq: zeta temperature only}.
\subsection{Stochastic linear problem with non-homogeneous Dirichlet boundary conditions}
\label{subsec: stochastic linear problem with noise on the boundary}
 We start this section by proving the first part of Proposition \ref{prop: regularity for Z_t}, in  particular the well-posedness and the regularity for the mild solution $Z_t^\varepsilon$ of the linear problem \eqref{eq: main results stochastic heat with noise on the boundary} subject to Dirichlet noise boundary conditions. Note that this result is well-known in the literature, see for instance \cite[Proposition 3.1]{DaPrato1993}.

\begin{proof}[Proof of Proposition \ref{prop: regularity for Z_t} (i)-(ii)]
Since (i)-(ii) do not depend on $\varepsilon$, we assume $\varepsilon = 1$ and omit the dependence on $\varepsilon$ of $\xi$ and $Z$ in the following.

(i) First, considering the splitting 
\[
Z_t =  (- \Delta) ^{\alpha_\theta} \int_0^t (-\Delta) ^{1-\alpha_\theta -\beta_\theta}e^{(t-r)\Delta }(-\Delta )^{\beta_\theta} DdW_r,
\]
$Z_t$ is well-defined and takes values in $H^{-2\alpha_\theta}(\mathcal{D})$, assuming \eqref{eq: convergence series eigenvalue noise for continuity xi}, if
\begin{equation}
    \left \lbrace
    \begin{aligned}
        \alpha_\theta + \beta_\theta & > \frac{1}{2}, \\
    \beta_\theta     & < \frac{1}{4}.
    \end{aligned}
    \right.
    \label{eq: conditions alpha e beta}
\end{equation}
Indeed, the first inequality in \eqref{eq: conditions alpha e beta} is a consequence of the square integrability conditions needed to define the It\^{o} integral. The second guarantees that
\[
\int_0^t (-\Delta)^{1-\alpha_\theta - \beta_\theta} e^{(t-r)\Delta} (-\Delta) ^{\beta_\theta }D  dW_r  = \sum_k \lambda_k \int_0^t (-\Delta) ^{1-\alpha_\theta - \beta_\theta }e^{(t-r)\Delta} (-\Delta) ^{\beta_\theta} D e_k \, d\beta_k (r)
\]
is well-defined. In particular $D e_k \in D((-\Delta)^{\beta_\theta})$, where the explicit characterisation of the domain for fractional powers of the Dirichlet Laplacian was recalled in Section \ref{subsec: notations and functional setting}. Note that \eqref{eq: conditions alpha e beta} is satisfied by our choice of $\alpha_\theta,\beta_\theta .$

(ii) We observe that $\xi_t$ is a Gaussian process with values in $L^2(\mathcal{D})$ and with zero mean. If we are able to prove that there exist $M>0$ and $\eta\in (0,1]$ such that
    \begin{equation}
        \mathbb{E} \left[\norm{\xi_{t_2} - \xi_{t_1}}_2^2  \right]\leq M (t_2 -t_1)^\eta, \quad \forall \, 0 \leq t_1 \leq t_2,
        \label{eq: kolmogorov test hp for gaussian processes}
    \end{equation}
    then the thesis follows from the Kolmogorov test for Gaussian processes, see \cite[Proposition 3.16]{DaPrato}. Consider
    \[
    \quad \xi_t =   \int_0^t (-\Delta) ^{1-\alpha_\theta}   e^{(t-r)\Delta }D dW_r =  \int_0^t (-\Delta) ^{1-\alpha_\theta - \beta_\theta}e^{(t-r)\Delta } (-\Delta)^{\beta_\theta }D dW_r,
    \]
    with
    \[
\alpha_\theta = \frac{1}{4}+ \frac{\delta_\theta}{2}, \quad \beta_\theta  = \frac{1}{4}- \frac{\delta_\theta}{4}.
\]
Then, for any $0 \leq t_1 \leq t_2$, we have
    \[
    \xi_{t_2} - \xi_{t_1} =  I_1 + I_2,
    \]
    with
    \[
    I_1 := \sum_{k}\lambda_{k}\int_{t_1}^{t_2}\left[
(-\Delta)^{1-\alpha_\theta-\beta_\theta}e^{\left(  t_2-r\right)  \Delta}(-\Delta)^{\beta_\theta }D e_{k}\right] \, d\beta _{k}(r)
\]
and
\[
I_2 :=  \sum_k \lambda_k\int_0^{t_1} \left[(-\Delta)^{1-\alpha_\theta - \beta_\theta } \left(e^{(t_2 -r)\Delta } - e^{(t_1-r) \Delta} \right)  (-\Delta)^{\beta_\theta} D e_k  \right]d \beta_k (r).
    \]
    
    For the estimate of the first integral $I_1$, by the independence of $(\beta_k) _k$, the It\^{o}-isometry, and the estimate for the heat-semigroup, we get
    \begin{equation}
        \begin{aligned}
                \mathbb{E} \left[ \| I_1\|_2^2 \right] &\leq \sum_k \lambda_k^2 \int_{t_1}^{t_2} \norm{(-\Delta)^{1-\alpha_\theta-\beta_\theta}e^{\left(  t_2-r\right)  \Delta} (-\Delta)^{\beta_\theta}D e_{k}}_2^2 dr   \\
                &\leq C\sum_k \lambda_k^2 \norm{(-\Delta)^{\beta_\theta} D e_k}_2^2\int_{t_1}^{t_2}  (t_2-r)^{-2(1-\alpha_\theta-\beta_\theta )}dr.
        \end{aligned}
    \end{equation}
     By our choice of $\alpha_\theta $ and $\beta_\theta$, it holds
    \[
    -2(1- \alpha_\theta - \beta_\theta) = -1+ \frac{\delta_\theta}{2}.
    \]
    Since
    \[
    \int_{t_1}^{t_2} (t_2-r)^{-1+ \frac{\delta_\theta }{2}} \,dr = \frac{2}{\delta_\theta  }(t_2 - t_1)^{{\delta_\theta }/2},
    \]
    we get, up to renaming $C$,
    \begin{equation}
        \begin{split}
                \mathbb{E}\left[\| I_1 \|_2^2 \right] \leq C (t_2-t_1)^{\frac{\delta_\theta }{2}}   \sum_k \lambda_k^2 \norm{(-\Delta)^{\beta_\theta} D e_k}_2^2 .
        \end{split}
        \label{eq: final estimate I1}
    \end{equation}
For the estimate of the second integral $I_2$, fix $0\leq t_1\leq t_2$. By the semigroup property,
\[
e^{(t_2-r) \Delta} -e^{(t_1-r)\Delta}=e^{(t_1-r) \Delta}(e^{(t_2-t_1)\Delta}-I),\qquad 0\le r\le t_1.
\]
Let $\gamma\in(0,1]$ (to be chosen later). Using the commutation of $(-\Delta)^\sigma$
with $e^{t \Delta}$ and inserting $(-\Delta)^\gamma (-\Delta)^{-\gamma}$, we obtain
\begin{equation}
    \begin{split}
        \norm{(-\Delta)^{1-\alpha_\theta-\beta_\theta}( e^{(t_2-r) \Delta}-  e^{(t_1-r) \Delta})}_{\mathcal L(L^2)}
&=\|(-\Delta)^{1-\alpha_\theta-\beta_\theta}e^{(t_1-r) \Delta}(e^{(t_2-t_1)\Delta}-I)\|_{\mathcal L(L^2)}\\
&=\|(-\Delta)^{1-\alpha_\theta-\beta_\theta+\gamma}e^{(t_1-r) \Delta} 
(-\Delta)^{-\gamma}(e^{(t_2-t_1) \Delta}-I)\|_{\mathcal L(L^2)}\\
&\leq \|(-\Delta)^{1-\alpha_\theta-\beta_\theta+\gamma}e^{(t_1-r) \Delta}  \| \, \| 
(-\Delta)^{-\gamma}(e^{(t_2- t_1)\Delta}-I)\|_{\mathcal L(L^2)}.
    \end{split}
    \label{eq:I2_split}
\end{equation}
For the first factor we use \cite[Theorem 6.13(c)]{Pazy} and we get
\begin{equation}
\|(-\Delta)^{1-\alpha_\theta-\beta_\theta+\gamma}e^{(t_1-r)\Delta}\|_{\mathcal L(L^2)}
\leq C (t_1-r)^{-(1-\alpha_\theta-\beta_\theta+\gamma)},
\label{eq: fractional power estimate for bound I2}
\end{equation}
for a constant $C$ independent of $t_1$. For the second factor, let $g\in L^2(\mathcal D)$ and set $x:=(-\Delta)^{-\gamma}g\in D((-\Delta)^\gamma)$. Since $(-\Delta)^{-\gamma}$ commutes with the heat semigroup $e^{(t_2- t_1) \Delta}$, we have
\[
(-\Delta)^{-\gamma}(e^{(t_2-t_1)\Delta}-I)g=(e^{(t_2-t_1) \Delta}-I)(-\Delta)^{-\gamma}g=(e^{(t_2-t_1) \Delta}-I)x.
\]
Applying \cite[Theorem 6.13(d)]{Pazy} with exponent $\gamma$ yields
\[
\|(e^{(t_2-t_1) \Delta}-I)x\|_2\leq C (t_2-t_1)^\gamma \|(-\Delta)^\gamma x\|_2 = C (t_2- t_1)^\gamma\|g\|_2,
\]
for a new constant $C$ independent of $t_1, t_2$. Therefore
\begin{equation}
\|(-\Delta)^{-\gamma}(e^{(t_2- t_1) \Delta}-I)\|_{\mathcal L(L^2)}
\le C (t_2- t_1)^\gamma.
\label{eq:I2_second_factor}
\end{equation}
Combining \eqref{eq:I2_split}-\eqref{eq:I2_second_factor} we obtain, for all $0\leq r\le t_1$,
\begin{equation}
\|(-\Delta)^{1-\alpha_\theta-\beta_\theta}(e^{(t_2-r)\Delta}-e^{(t_1-r) \Delta})\|_{\mathcal L(L^2)}
\leq C (t_2- t_1)^\gamma (t_1-r)^{-(1-\alpha_\theta-\beta_\theta+\gamma)}.
\label{eq:pazy_difference_bound}
\end{equation}

Using independence of $(\beta_k)_k$, It\^o isometry, and \eqref{eq:pazy_difference_bound}, we get
\begin{align*}
\mathbb{E} \|I_2\|_2^2 
&=\sum_k \lambda_k^2\int_0^{t_1}
\|(-\Delta)^{1-\alpha_\theta-\beta_\theta}(e^{(t_2-r) \Delta}-e^{(t_1-r) \Delta})(-\Delta)^{\beta_\theta}De_k\|_2^2\,dr\\
&\leq \sum_k \lambda_k^2\int_0^{t_1}
\|(-\Delta)^{1-\alpha_\theta-\beta_\theta}(e^{(t_2-r) \Delta}-e^{(t_1-r) \Delta})\|_{\mathcal L(L^2)}^2 
\|(-\Delta)^{\beta_\theta}De_k\|_2^2\,dr\\
&\leq C (t_2- t_1)^{2\gamma}\left(\sum_k \lambda_k^2\|(-\Delta)^{\beta_\theta}De_k\|_2^2\right)
\int_0^{t_1}(t_1-r)^{-2(1-\alpha_\theta-\beta_\theta+\gamma)}\,dr.
\end{align*}
The time-integral is finite provided $2(1-\alpha_\theta-\beta_\theta+\gamma)<1$, i.e.
\[
\gamma<\alpha_\theta+\beta_\theta-\frac{1}{2}.
\]
Since $\alpha_\theta+\beta_\theta=\frac{1}{2}+\frac{\delta_\theta}{4}$, we may choose
any $\gamma\in (0,\tfrac{\delta_\theta}{4})$, and then
\begin{equation}
\mathbb E\|I_2\|_2^2
\leq C\,(t_2-t_1)^{2\gamma}\sum_k \lambda_k^2\|(-\Delta)^{\beta_\theta}De_k\|_2^2,
\label{eq: finale estimate I2}
\end{equation}
where $C>0$ depends on $T,\alpha_\theta,\beta_\theta,\gamma$ but not on $t_1,t_2$. In conclusion, applying Jensen's inequality and setting $\eta := \min (\delta_\theta/2, 2\gamma ) \in (0,1]$, we obtain \eqref{eq: kolmogorov test hp for gaussian processes} by combining the estimates for $I_1$, $I_2$ in \eqref{eq: final estimate I1} and \eqref{eq: finale estimate I2}.

\end{proof}

We conclude this section by recalling, and then applying, the notion of stochastic convolution to prove point (iii) of Proposition \ref{prop: regularity for Z_t}. 

Given $U,H$, Hilbert spaces, let $U_0 := \mathcal{Q}^{1/2}(U)$ and set $L^0_2 = L_2(U_0,H)$ the space of Hilbert-Schmidt operators from $U_0$ to $H$, with norm
\[
\norm{R}_{L^0_2} = \norm{ R \circ \mathcal{Q}^{1/2}}_{L_2(U,H)},
\]
where $L_2 = L_2(U,H)$ is the space of Hilbert-Schmidt operators from $U$ to $H$ with norm
\[
\norm{R}_{L_2}^2 = \sum_k \norm{Re_k}_H^2, 
\]
where again $(e_k)_k \subset U$ denotes an orthonormal basis of $U$. Consider $(W_t)_t$ to be a $U$-valued $\mathcal{Q}$-Wiener process and $\mathcal{A} \colon D(\mathcal{A}) \subset H \to H$ a linear operator which generates a $C_0$-semigroup $(e^{t \mathcal{A}})_{t \geq 0}$ in $H$. Then, 
we consider the following space of stochastic processes
\[
\mathcal{N}_W^2(0,T) := \left \lbrace  \Phi \colon [0,T] \times \Omega  \to L^0_2 \; \mid \; \Phi \text{ is predictable  and }\norm{\Phi}_{T} <\infty \right \rbrace ,
\]
where
\[
\norm{\Phi}_{T} := \left( \mathbb{ E}  \left[ \int_0^{T} \norm{\Phi (t)}_{L^0_2}^2 \, dt \right]\right)^{\frac{1}{2}}.
\]
Lastly, for $\Phi \in \mathcal{N}^2_W(0,T)$ we recall that the stochastic convolution $W_{\mathcal{A}}^\Phi$ is defined as
\[
W_\mathcal{A}^\Phi (t) = \int_0^t e^{(t-r)\mathcal{A}} \Phi (r)\,dW_r, \quad t \in [0,T],
\]
and we refer to \cite{DaPrato} for more information. We can now move to the proof of the last part of Proposition \ref{prop: regularity for Z_t}.

\begin{proof}[Proof of Proposition \ref{prop: regularity for Z_t} (iii)]
 By Markov's inequality, for any $\vartheta>0$, we have
    \begin{equation}
        \begin{split}
                \mathbb{P} \left( \sup_{ 0 \leq t \leq T} \norm{Z_t^\varepsilon}_{H^{-2 \alpha_\theta}(\mathcal{D})} > \vartheta  \right)& \leq \frac{1}{\vartheta ^2} \mathbb{E} \left[ \sup_{ 0 \leq t \leq T} \norm{Z_t^\varepsilon}^2_{H^{-2\alpha_\theta}(\mathcal{D})}  \right].
        \end{split}
        \label{eq: markov inequality prop 21}
    \end{equation}
    To obtain the thesis, we thus need to bound the expectation on the right-hand side. We start to observe that, since $Z_t^\varepsilon = (-\Delta)^{ \alpha_\theta}\xi_t^\varepsilon$, then
    \[
    \norm{Z^\varepsilon_t}_{H^{-2\alpha_\theta}(\mathcal{D})} =   \norm{(-\Delta)^{ \alpha_\theta}\xi_t^\varepsilon}_{H^{-2\alpha_\theta}(\mathcal{D})} \cong \norm{\xi_t^\varepsilon}_2.
    \]
    Second, recall that
    \[
    \xi_t^\varepsilon =  \sqrt{\varepsilon} \int_0^t (- \Delta)^{1- \alpha_\theta - \beta_\theta } e^{(t-r) \Delta } (- \Delta)^{\beta_\theta} D\, dW_r.
    \]
    Set $\Phi := (-\Delta)^{\beta_\theta} D $, and observe that $\Phi \in L^0_2$ since 
    \[
    \| \Phi \|_{L^0_2}^2 = \sum_k \lambda_k^2 \| (-\Delta)^{\beta_\theta }D e_k \|_2^2 <\infty,
    \]
    where $(e_k)_k$ is an orthonormal basis of $U $ made of eigenvectors for the covariance operator $\mathcal{Q}.$ Thus, it is well-defined the $L^2(\mathcal{D})$-stochastic convolution
    $
    W^\Phi_\Delta (t):= \int_0^t e^{(t-r) \Delta }\Phi \, dW_r
    $, for $t \in [0,T]$, and $\xi_t^\varepsilon = \sqrt{\varepsilon} (- \Delta)^{\frac{1}{2}- \frac{\delta_\theta }{4}} W_\Delta ^\Phi (t)$. Choose now $\rho \in ( \frac{1}{2}- \frac{\delta_\theta }{4}, \frac{1}{2})$, and set
    \[
    Y_\rho (t) := \int_0^t (t-r)^{-\rho} e^{(t-r) \Delta } \Phi dW_r, \quad t \in [0,T].
    \]
    By the factorisation method, see \cite[Section 5.3.1]{DaPrato}, since $\int_0^T t^{-2\rho} \|e^{t \Delta} \Phi \|_{L^0_2}^2 dt <\infty $,
    it holds
    \[
    W^\Phi_\Delta (t) =\frac{\sin (\pi \rho)}{\pi } \int_0^t (t-r)^{\rho-1} e^{(t-r) \Delta } Y_\rho (r) \, dr.
    \]
Thus, using the estimate for the fractional powers as in \eqref{eq: fractional power estimate for bound I2}, we have
\begin{equation*}
    \begin{split}
        \| \xi_t^\varepsilon\|_2 \leq C \sqrt{\varepsilon} \int_0^t (t-r)^{\rho -1} \| (- \Delta)^{\frac{1}{2}- \frac{\delta_\theta }{4}} e^{(t-r)\Delta}\|_{\mathcal{L}(L^2)} \| Y_\rho (r) \|_2 dr \leq C \sqrt{\varepsilon} \int_0^t (t-r)^{\rho-\frac{3}{2}+ \frac{\delta_\theta }{4}} \| Y_\rho (r) \|_2 dr.
    \end{split}
\end{equation*}
    Fix now $q >2 $ such that $\frac{1}{q}< \rho - (\frac{1}{2}- \frac{\delta_\theta}{4})$, and let $q '$ be the conjugate exponent of $q $, i.e. $q' = \frac{q}{q -1}.$ Then $g(\sigma ):= \sigma ^{\rho -1-(\frac{1}{2}-\frac{\delta_\theta }{4})} \in L^{q '}(0, T)$, and by H\"older inequality
    \[
    \sup_{0 \leq t \leq T} \| \xi^\varepsilon_t \|_2 \leq C \sqrt{\varepsilon} \| g \|_{L^{q'}(0,T)} \| Y_\rho \|_{L^q (0,T;L^2(\mathcal{D}))}.
    \]
    Considering the $q$-th moment, we have
    \begin{equation}
        \mathbb{E} \left[\sup_{0 \leq t \leq T} \| \xi_t ^\varepsilon \|_2^q  \right] \leq C \varepsilon^{q /2} \mathbb{E} \left[  \int_0^T \| Y_\rho (r) \|_2^q \, dr \right].
        \label{eq: gamma moment start xi}
    \end{equation}
    Since $Y_\rho$ is a centered Gaussian random variable in $L^2(\mathcal{D})$, there exists $c_q >0$ such that $\mathbb{E}[ \| Y_\rho (r) \|_2^q] \leq c_q\left( \mathbb{E}[\| Y_\rho (r) \|_2^2]\right)^{q/2}$. By It\^{o} isometry and the contraction of $e^{t \Delta}$ on $L^2$, we have
    \[
    \mathbb{E} \left[\| Y_\rho (r)\|_2^2\right] = \int_0^r (r- \sigma )^{- 2\rho } \| e^{(r- \sigma ) \Delta } \Phi \|_{L^0_2}^2 \, d\sigma  \leq \| \Phi \|_{L^0_2}^2 \int_0^r (r- \sigma)^{- 2\rho}\, d \sigma = \| \Phi \|_{L^0_2}^2 \frac{r^{1- 2 \rho}}{1- 2\rho } .
    \]
    Therefore
    \begin{equation}
\mathbb{E}\left[ \int_0^T \| Y_\rho (r) \|_2^{q} \, dr\right] \leq c_q \| \Phi \|_{L^0_2}^q\int_0^T r^{\frac{(1-2\rho) q }{2}} dr \leq c_{q, \rho, T} \| \Phi \|_{L^0_2}^q. 
        \label{eq: bound Y rho  gamma moment}
    \end{equation}
    Thus, substituting \eqref{eq: bound Y rho  gamma moment} inside \eqref{eq: gamma moment start xi}, we obtain
    \[
    \mathbb{ E} \left[ \sup_{0 \leq t \leq T} \| \xi_t^\varepsilon \|_2^q \right] \leq C(\delta_\theta , T, q) \varepsilon^{q /2 } \| \Phi \|_{L^0_2}^q .
    \]
    Lastly, since $q >2$, by Jensen inequality we deduce
    \begin{equation*}
    \begin{split}
    \mathbb{E} \left[\sup_{0 \leq t \leq T} \| \xi_t^\varepsilon \|_2^2 \right] \leq \left( \mathbb{E} \left[ \sup_{0 \leq t \leq T} \| \xi_t^\varepsilon \|_2^q\right]\right)^{2/q} \leq \varepsilon C(\delta_\theta, T)  \| \Phi \|_{L^0_2}^2 
    =  \varepsilon C(\delta_\theta, T) \sum_k \lambda_k^2 \| (- \Delta)^{\beta_\theta } De_k \|_2^2.
    \end{split}
    \end{equation*}
    The thesis then follows by substituting the previous bound into \eqref{eq: markov inequality prop 21}.
\end{proof}

\subsection{The remainder temperature equation}
Let $\delta_u >0$ be a small regularity exponent for the velocity. Given a prescribed divergence-free velocity field 
\[
u \in L^p(0,T; H^{\frac{3}{2}-\delta_u}_\sigma (\mathcal{D})) \cap L^\infty (0,T; H^{\frac{1}{2}-\delta_u}_\sigma (\mathcal{D})), \quad p\in(1,\infty)
\]
in this section we consider the remainder $\zeta_t^\varepsilon:= \theta^\varepsilon_t - Z^\varepsilon_t$ between the solution $\theta_t^\varepsilon$ of the system \eqref{eq: temperature for given velocity}, and the solution $Z^\varepsilon_t$ of the linear problem \eqref{eq: main results stochastic heat with noise on the boundary} considered in the previous section.

We start by proving an auxiliary interpolation result for the velocity field $u$, which will be used throughout this section.
\begin{corollary}
Let $\delta_u \in (0,1)$ and $p\in(1,\infty)$. If 
\[
u \in L^p(0,T; H^{\frac{3}{2}-\delta_u}(\mathcal{D})) \cap L^\infty(0,T; H^{\frac{1}{2}-\delta_u}(\mathcal{D})),
\]
then for all $ 0 < \lambda\leq 1-\delta_u$, we have
\[
u\in L^{\frac{p}{\lambda+\delta_u}}\left(  0,T;H^{\frac{1}{2}+\lambda}(\mathcal{D})\right).
\]
    \label{cor: interpolation regularity of u}
\end{corollary}
\begin{proof}
Note that if $\lambda = 1- \delta_u$, then the statement is trivial. Thus, fix $0 < \lambda < 1- \delta_u$, and define $\delta' := 1- \lambda.$ Then $ \delta_u < \delta ' < 1+ \delta_u $. Since $\frac{3}{2}- \delta' \in (\frac{1}{2}- \delta_u, \frac{3}{2}- \delta_u )$, there exists $\vartheta  \in (0,1)$ such that
\[
\frac{3}{2}-\delta' =\left(  1-\vartheta\right) \left(  \frac{1}
{2}-\delta_u\right)  +\vartheta \left(  \frac{3}{2}-\delta_u\right),
\]
which yields
\[
\vartheta = 1- \delta'+ \delta_u.
\]
Using the interpolation inequality
\[
\left\Vert u_t  \right\Vert _{H^{\frac{3}{2}-\delta^{\prime
}}(\mathcal{D})}\leq C\left\Vert u_t  \right\Vert _{H^{\frac{1}{2}-\delta_u}%
(\mathcal{D})}^{\delta'-\delta_u} \Vert u_t\Vert
_{H^{\frac{3}{2}-\delta_u}(\mathcal{D})}^{1-\delta'+\delta_u},
\]
for a constant $C>0$. Therefore, for any $q \geq 1$, we have
\begin{align*}
\int_{0}^{T}\left\Vert u_t  \right\Vert _{H^{\frac{3}
{2}-\delta'}(\mathcal{D})}^q dt  & \leq C\int_{0}^{T}\left\Vert u_t
\right\Vert _{H^{\frac{1}{2}-\delta_u} (\mathcal{D})}^{\left(  \delta'%
-\delta_u\right)  q}\left\Vert u_t  \right\Vert _{H^{\frac
{3}{2}-\delta_u}(\mathcal{D})}^{\left(  1-\delta'+\delta_u\right)  q}dt\\
& \leq C\left\Vert u\right\Vert _{L^{\infty}\left( 0,T; H^{\frac{1}
{2}-\delta_u}(\mathcal{D})\right)  }^{\left(  \delta'-\delta_u\right)  q}  \left\Vert
u\right\Vert _{L^{\left(  1-\delta'+\delta_u\right)  q}\left(0,T;
H^{\frac{3}{2}-\delta_u}(\mathcal{D})\right)  }^{\left(  1-\delta'+\delta_u\right)
q}.
\end{align*}
This is finite if $  (1-\delta'+\delta_u)   q=p$, which gives
\[
q=\frac{p}{1-\delta'+\delta_u} = \frac{p}{\lambda + \delta_u}.
\]
Further, observe that $ \frac{3}{2}-\delta ' = \lambda + \frac{1}{2}$ by definition of $\lambda$. Thus, we have proved 
\[
\int_0^T \norm{u_t}_{H^{\lambda + \frac{1}{2}}(\mathcal{D})}^{\frac{p}{\lambda + \delta_u}} < \infty,
\]
which gives the thesis.
\end{proof}

To make sense of the mild solution, we next establish a space-time integrability result for the product $u_t \cdot  Z_t^\varepsilon$. Recall that, for any $\delta_\theta >0$, $\alpha_\theta = \frac{1}{4}+ \frac{\delta_\theta}{2}$ and $\beta_\theta = \frac{1}{4}-\frac{\delta_\theta}{4}$, we have introduced the stochastic processes 
\[
Z_t^\varepsilon = (-\Delta)^{\alpha_\theta} \xi_t^\varepsilon, 
\quad 
\xi_t^\varepsilon = \int_0^t (-\Delta)^{1-\alpha_\theta} e^{(t-r)\Delta} D dW_r.
\]
Further, if  
\[
\sum_k \lambda_k^2  \norm{(-\Delta)^{\beta_\theta} D e_k}_2^2 <\infty,
\]
then by Proposition \ref{prop: regularity for Z_t} we know that the trajectories of $
\xi^{\varepsilon}$ are in $C(0,T; L^2(\mathcal{D}))$.
\begin{lemma} \label{lemma: regolarita}
Let $0 < \lambda < 1$ and $\gamma > \frac{1}{4}$. Set
\[
\bar \delta_\theta (\lambda, \gamma) := \min \{ \lambda, 2\gamma - \frac{1}{2}\}.
\]
For any $0< \delta_\theta <  \bar{ \delta}_\theta $, any velocity field 
\[
u \in L^p(0,T; H^{\frac{3}{2}-\delta_u}(\mathcal{D})) \cap L^\infty (0,T;H^{\frac{1}{2}-\delta_u}(\mathcal{D})), \qquad p >1, \, \delta_u \in (0, 1- \lambda],
\]
we have 
\[
t\mapsto (-\Delta)^{-\gamma- \frac{1}{2}} \mathrm{div}\left(  u_{t}(-\Delta)^{\alpha_\theta }\xi_{t}^\varepsilon\right) \in L^{\frac{p}{\lambda+\delta_u}}\left( 0,T;L^{\frac{6}{5}}( \mathcal{D})\right), \qquad \alpha_\theta  = \frac{1}{4}+ \frac{\delta_\theta }{2}.
\]
Moreover,
\begin{equation}
\norm{(-\Delta)^{-\gamma- \frac{1}{2}} \mathrm{div}( u_t (-\Delta)^{\alpha_\theta} \xi_t^\varepsilon)}_{L^\frac{p}{\lambda + \delta_u }(0,T;L^{\frac{6}{5}}(\mathcal{D}))} \leq  C \norm{\xi^\varepsilon}_{C(0,T; L^2(\mathcal{D}))} \, \norm{u}_{L^{\frac{p}{\lambda + \delta_u }}(0,T; W^{\frac{1}{2}+\lambda,2}(\mathcal{D}))},
    \label{eq: norm estimate u Delta xi}
\end{equation}
where $C = C(\lambda,\gamma,\delta_\theta)>0$.
\end{lemma}
\begin{proof}
To simplify the notation, we set $\varepsilon = 1$ and drop the dependence on $\varepsilon$ of $\xi^\varepsilon$.
We start by recalling that, by \cite[Theorem A.1]{Behzadan} (which we can apply after extending functions on $\mathcal D$ to $\mathbb R^3$ via a bounded extension operator), the pointwise multiplication
\begin{equation*}
    \begin{aligned}
        W^{s_{1},p_{1}}\left(\mathcal{D} \right)\times W^{s_{2},p_{2}} \left(\mathcal{D} \right)
        &\hookrightarrow W^{s,p} \left(\mathcal{D} \right)\\
        (f,g) & \mapsto f g
    \end{aligned}
\end{equation*}
is bilinear and continuous if
\begin{enumerate}
    \item[(i)] $s_1 \geq s$ and $s_2 \geq s$;
    \item[(ii)] $s_{1}\geq s+3\left(\frac{1}{p_{1}}-\frac{1}{p}\right)$ and 
                $s_{2}\geq s+3\left(\frac{1}{p_{2}}-\frac{1}{p}\right)$;
    \item[(iii)] $s_{1}+s_{2}> s+3\left(  \frac{1}{p_{1}}+\frac{1}{p_{2}}-\frac{1} {p}\right)$;
    \item[(iv)] if $s_i=s\notin\mathbb N$ for some $i\in\{1,2\}$, then one must additionally assume $p_i\le p$.
\end{enumerate}

\textbf{Step 1}. As a first step of the proof, we show that
\[
\left\langle (-\Delta)^{-\gamma- \frac{1}{2}} 
\mathrm{div}\left(  u_{t}(-\Delta)^{\alpha_\theta }\xi_{t}\right),\phi\right\rangle 
\leq K_{t}\left\Vert \phi\right\Vert _{L^{6}\left(\mathcal{D} \right)},
\]
where the integrability in time of $t \mapsto K_t$ will be investigated in Step $2$. Here, $6$ is the conjugate exponent of $\frac{6}{5}$, $\phi \in L^6(\mathcal{D})$ is a test function, and $\alpha_\theta = \frac{1}{4}+ \frac{\delta_\theta}{2}$. Further, we denote by $\langle \cdot, \cdot \rangle$ the duality pairing. Since $\Delta$ is self-adjoint, by H\"older inequality and integration by parts, we obtain
\begin{align*}
\left\langle (-\Delta)^{-\gamma- \frac{1}{2}} \mathrm{div} 
\left(u_{t}(-\Delta)^{\alpha_\theta }\xi_{t}\right),\phi\right\rangle  
&  =\left\langle \mathrm{div}(u_{t}(-\Delta)^{\alpha_\theta}\xi_{t}),
(-\Delta)^{-\gamma- \frac{1}{2}}\phi\right\rangle \\
& = - \langle  u_t (-\Delta)^{\alpha _\theta}\xi_t, 
\nabla (-\Delta) ^{-\gamma - \frac{1}{2}} \phi \rangle \\
&  =-\left\langle (-\Delta)^{\alpha_\theta}\xi_{t},
u_{t} \cdot  \nabla (-\Delta)^{-\gamma- \frac{1}{2}}\phi\right\rangle \\
&  =-\left\langle \xi_{t},
(-\Delta)^{\alpha_\theta} \left(u_t \cdot\nabla  (-\Delta)^{-\gamma- \frac{1}{2}} \phi \right)\right\rangle\\
&  \leq\left\Vert \xi_{t}\right\Vert _{L^{2}\left(\mathcal{D} \right)}
\left\Vert (-\Delta)^{\alpha_\theta}
\left(u_{t} \cdot \nabla (-\Delta)^{-\gamma- \frac{1}{2}}\phi\right)  
\right\Vert _{L^{2}\left(\mathcal{D} \right)}\\
&  =\left\Vert \xi_{t}\right\Vert _{L^{2}\left(\mathcal{D} \right)}
\left\Vert u_{t}\cdot \nabla (-\Delta)^{-\gamma - \frac{1}{2}}\phi\right\Vert 
_{W^{2\alpha_\theta,2}\left(\mathcal{D} \right)}.
\end{align*}
Next, we show that 
$u_{t} \cdot \nabla (-\Delta)^{-\gamma - \frac{1}{2}}\phi\in 
W^{2\alpha_\theta,2}\left(\mathcal{D} \right)$. Since 
$2\alpha_\theta=\frac{1}{2}+\delta_\theta$, this is equivalent to
\begin{equation}
    u_{t}\cdot \nabla (-\Delta)^{-\gamma - \frac{1}{2}}\phi\in 
    W^{\frac{1}{2}+\delta_\theta,2}\left(\mathcal{D} \right).
    \label{eq: embedding auxiliary 1}
\end{equation}
By the norm equivalence between $W^{\frac{1}{2}+\lambda,2 }(\mathcal{D})$ 
and $H^{\frac{1}{2}+\lambda} (\mathcal{D})$, and by Corollary 
\ref{cor: interpolation regularity of u}, we have 
$u_{t}\in W^{\frac{1}{2}+\lambda,2}\left(\mathcal{D} \right)$ a.e., 
and since $\phi\in L^{6}\left(\mathcal{D} \right)$, we deduce 
$\nabla (-\Delta)^{-\gamma- \frac{1}{2}}\phi
\in W^{2\gamma,6}\left(\mathcal{D} \right)$. Thus 
\eqref{eq: embedding auxiliary 1} follows from the pointwise multiplication embedding 
\begin{equation*}
    \begin{aligned}
        W^{\frac{1}{2}+\lambda,2}\left(\mathcal{D} \right)
        \times W^{2\gamma,6}\left(\mathcal{D} \right)& 
        \hookrightarrow W^{\frac{1}{2}+\delta_\theta,2}\left(\mathcal{D} \right) \\
        (u_t, \nabla (-\Delta)^{-\gamma- \frac{1}{2}} \phi ) 
        &\mapsto u_t \cdot \nabla (-\Delta)^{-\gamma - \frac{1}{2}}
 \phi,    \end{aligned}
\end{equation*}
provided that the following conditions hold
\begin{enumerate}
    \item[(i)] $\frac{1}{2}+\lambda\geq\frac{1}{2}+\delta_\theta$ 
               and $2\gamma\geq\frac{1}{2}+\delta_\theta$;
    \item[(ii)] $\frac{1}{2}+\lambda\geq\frac{1}{2}+\delta_\theta
                 +3\left(  \frac{1}{2}-\frac{1}{2}\right)$ and 
                 $2\gamma\geq\frac{1}{2}+\delta_\theta
                 +3\left(  \frac{1}{6}-\frac{1}{2}\right)$;
    \item[(iii)] $\frac{1}{2}+\lambda+2\gamma >\frac{1}{2}+\delta_\theta
         +3\left(\frac{1}{2}+\frac{1}{6}-\frac{1}{2}\right)$.
\end{enumerate}
The previous conditions are equivalent to
\begin{enumerate}
    \item[(a)] $\delta_\theta\leq \lambda $ and 
               $ \delta_\theta \leq 2 \gamma - \frac{1}{2}$;
    \item[(b)] $ \delta_\theta \leq \frac{1}{2}+ 2 \gamma$;
    \item[(c)] $\delta_\theta < \lambda + 2\gamma - \frac{1}{2}.$
\end{enumerate}

Since in general $s=\frac{1}{2}+\delta_\theta\notin\mathbb N$ and here $p_2=6>p=2$, condition (iv) in
\cite[Theorem A.1]{Behzadan} forbids the borderline case $s_2=s$. Hence we impose
\[
s_2  > s \quad \Longleftrightarrow \quad 2\gamma> \frac{1}{2}+ \delta_\theta
\quad \Longleftrightarrow \quad \delta_\theta<2\gamma-  \frac{1}{2}.
\]
In contrast, the endpoint $\delta_\theta=\lambda$ is admissible since then
$s_1=s$ but $p_1=p=2$.

Therefore, we can choose
\begin{equation}
\bar{\delta}_\theta (\lambda, \gamma):= \min \left \lbrace  
\lambda, 2 \gamma - \frac{1}{2}\right \rbrace >0,
\label{eq: definiton of bar delta in Lemma 33}
\end{equation}
and note that $\bar{\delta}_\theta$ is positive since $\gamma> \frac{1}{4}$ 
and $\lambda \in (0,1).$ 

\textbf{Step 2}. Fix now $0<\delta_\theta\leq \bar{\delta}_\theta (\lambda,\gamma)$. The computations of the previous step yield
\begin{align*}
    \left\Vert (-\Delta)^{-\gamma- \frac{1}{2}} 
    \mathrm{div}\left(u_{t} (-\Delta)^{\alpha_\theta}\xi_{t}\right) 
    \right\Vert_{L^\frac{6}{5} \left(\mathcal{D} \right)} 
    &= \sup _{\|\phi\|_{L^6(\mathcal{D}) } \leq 1} 
    \left\langle (-\Delta)^{-\gamma - \frac{1}{2}} 
    \mathrm{div}\left(u_{t}(-\Delta)^{\alpha_\theta}\xi_{t}\right),\phi\right\rangle \\
    &  \leq C\left\Vert \xi_{t}\right\Vert _{L^{2}\left(\mathcal{D} \right)} \sup _{\|\phi\|_{6 } \leq 1} 
    \left\Vert u_{t}\cdot \nabla (-\Delta)^{-\gamma- \frac{1}{2}}\phi\right\Vert _{W^{2\alpha_\theta,2}\left(\mathcal{D} \right)} \\
    &\leq C\left\Vert \xi_{t}\right\Vert_{L^{2}\left(\mathcal{D} \right)}
    \left\Vert u_t \right\Vert_{W^{\frac{1}{2}+\lambda,2}\left(\mathcal{D} \right)}  \sup _{\|\phi\|_{6 } \leq 1} 
    \left\Vert \nabla (-\Delta)^{-\gamma- \frac{1}{2}}\phi \right\Vert
    _{W^{2\gamma,6}\left(\mathcal{D} \right)}\\
    & \leq C \left\Vert \xi_{t}\right\Vert_{L^{2}\left(\mathcal{D} \right)}
    \left\Vert u_t \right\Vert_{W^{\frac{1}{2}+\lambda,2} \left(\mathcal{D} \right)} \sup _{\|\phi\|_{6 } \leq 1} 
    \left\Vert \phi \right\Vert_{L^{6}\left(\mathcal{D} \right)}\\
    & \leq  C \left\Vert \xi_{t}\right\Vert_{L^{2}\left(\mathcal{D} \right)}
    \left\Vert u_t \right\Vert_{W^{\frac{1}{2}+\lambda,2} \left(\mathcal{D} \right)} ,
\end{align*}
where $C = C(\lambda, \gamma, \delta_\theta)$ is a constant from 
the multiplication results. Thus, up to positive constants, we obtain from H\"older's inequality that
\begin{align*}
 \int_0^T 
 \left\Vert (-\Delta)^{-\gamma- \frac{1}{2}}
 \mathrm{div}\left(u_{t}(-\Delta)^{\alpha_\theta}\xi_{t}\right) 
 \right\Vert_{L^\frac{6}{5} \left(\mathcal{D} \right)}^{\frac{p}{\lambda+\delta_u}}dt  
 &\leq  \int_0^T 
 \left(\left\Vert \xi_{t}\right\Vert_{L^{2}\left(\mathcal{D} \right)}
        \left\Vert u_t \right\Vert_{W^{\frac{1}{2}+\lambda,2}\left(\mathcal{D} \right)} 
 \right)^{\frac{p}{\lambda+\delta_u}}dt\\
 & =   \int_0^T \left\Vert \xi_{t}\right\Vert_{L^{2}\left(\mathcal{D} \right)}^{\frac{p}{\lambda+\delta_u}}
        \left\Vert u_t \right\Vert_{W^{\frac{1}{2}+\lambda,2}\left(\mathcal{D} \right)}^{\frac{p}{\lambda+\delta_u}}dt \\
 & \leq   \sup_{0\leq t \leq T} \left\Vert \xi_{t}\right\Vert_{L^{2}\left(\mathcal{D} \right)}^{\frac{p}{\lambda+\delta_u}} 
        \int_0^T \left\Vert u_t \right\Vert_{W^{\frac{1}{2}+\lambda,2}\left(\mathcal{D} \right)}^{\frac{p}{\lambda+\delta_u}}dt .
\end{align*}
Note that all the terms on the right-hand side are finite thanks to 
Proposition \ref{prop: regularity for Z_t}, 
Corollary \ref{cor: interpolation regularity of u}, 
and the norm equivalence between $W^{\frac{1}{2}+\lambda,2 }(\mathcal{D})$ and 
$H^{\frac{1}{2}+\lambda} (\mathcal{D})$; the last corollary can be applied 
since by hypothesis $\delta_u \leq 1- \lambda$.
\end{proof}

Recall that the mild solution of the remainder's problem \eqref{eq: zeta temperature only} has the form 
\[
\zeta_{t}^\varepsilon=e^{t\Delta }\theta _{0}-\int_{0}^{t}e^{\left(  t-r\right)  \Delta}\left(
u_{r}^\varepsilon\cdot\nabla\zeta_{r}^\varepsilon\right)  dr-\int_{0}^{t}e^{\left(  t-r\right)
\Delta}\left(  u_{r}^\varepsilon\cdot\nabla Z_{r}^\varepsilon\right)  dr.
\]
By the previous auxiliary results, we can deduce the regularity of the third term on the right-hand side of the previous identity.

\begin{corollary} \label{cor: csi continuo}
 Let $s \in (0, \frac{1}{2})$, $p>1$ and $\gamma > \frac{1}{4}$ be given. Assume that
\begin{equation}
    a(s,\gamma):=\frac{1}{2}+ \gamma + \frac{s}{2} < 1.
    \label{eq: condition cor csi continuo}
\end{equation}
Let $\lambda  \in (0,1)$ such that
\begin{equation}
\lambda < p ( 1- a(s,\gamma)).
\label{eq: secon condition csi continuo}
\end{equation}
Fix $0 <  \delta_\theta <  \bar \delta _\theta = \min \{ \lambda, 2 \gamma - \frac{1}{2} \},$ and consider $\alpha_\theta := \frac{1}{4}+ \frac{\delta_\theta }{2}.$ Define
\[
\bar \delta _u( p, s, \gamma, \lambda) := \min \{ p (1-a(s, \gamma))- \lambda, 1- \lambda \} >0.
\]
Then, for every $\delta_u \in (0, \bar \delta_ u)$, and every velocity field 
\[
u\in L^p(0,T;H^{\frac32-\delta_u}_\sigma(\mathcal D))
\cap L^\infty(0,T;H^{\frac12-\delta_u}_\sigma(\mathcal D)),
\]
the function
\[
\chi_t:=\int_0^t e^{(t-r)\Delta}\bigl(u_r\cdot\nabla Z_r^\varepsilon\bigr)\,dr
\]
belongs to $C(0,T;W^{s,\frac65}(\mathcal D))$ and satisfies
\begin{equation}
\|\chi^\varepsilon\|_{C(0,T;W^{s,\frac65})}
\le C\,\|\xi^\varepsilon\|_{C(0,T;L^2)}\,
\|u\|_{L^{\frac{p}{\lambda+\delta_u}}(0,T;W^{\frac12+\lambda,2})},
\label{eq: estimate norm xi}
\end{equation}
where $C=C(s,p,\gamma,\lambda,\delta_\theta,T)>0$. Lastly, given $s \in (0, \frac{1}{2}),p >1$ as before, there always exists $\gamma = \gamma(s)> \frac{1}{4}$ and $\lambda = \lambda (s,p) \in (0,1)$ such that \eqref{eq: condition cor csi continuo} and \eqref{eq: secon condition csi continuo} hold.
\end{corollary}
\begin{proof}
To simplify the notation, we set $\varepsilon = 1$ and drop the dependence of $\varepsilon$ of $\xi^\varepsilon$.

\textbf{Step 1.} Fix $\delta_\theta >0$ such that $ 0 < \delta_\theta < \bar \delta_\theta = \min \{ \lambda , 2 \gamma - \frac{1}{2} \}$. Define the function
\[
f_\lambda  (\delta_u) = ( \gamma+ \frac{s}{2}+ \frac{1}{2}) \frac{1}{1- \frac{\lambda}{p}- \frac{\delta_u}{p}} = \frac{a(s, \gamma)}{1- \frac{\lambda + \delta_u}{p}}, \quad \lambda \geq 0.
\]
If $\delta_u < \bar \delta _u$, then $\lambda + \delta_u < p (1-a(s,\gamma))$, which implies 
\begin{equation}
     f_\lambda(\delta_u ) =\bigl(\gamma + \tfrac{s}{2} + \tfrac12\bigr)\,\frac{1}{1 - \frac{\lambda}{p} - \frac{\delta_u}{p}} < 1.
    \label{eq: auxiliary eq lambda delta}
\end{equation}

 \textbf{Step 2.} Thanks to the commuting properties of the fractional Laplacian with the heat semigroup and the divergence free condition of $u$, it holds
\begin{align*}
\chi_{t}  &  =(-\Delta)^{-\frac{s}{2}} \rho_{t},\\
\rho_{t}  &  =\int_{0}^{t} (-\Delta)^{\gamma+\frac{s}{2}+\frac{1}{2}}e^{\left(
t-r\right)  \Delta} (-\Delta)^{-\gamma- \frac{1}{2}}\mathrm{div}\left(  u_{r} (-\Delta)^{\alpha_\theta }\xi_{r}\right)  dr.
\end{align*}
By the estimate
\[
\norm{ (-\Delta)^\varsigma e^{(t-r)\Delta} y}_{L^\frac{6}{5}(\mathcal{D})} \leq \frac{c}{(t-r)^\varsigma} \norm{y}_{L^\frac{6}{5}(\mathcal{D})},
\]
where $\varsigma:= \gamma+\frac{s}{2}+\frac{1}{2}$ and $c=c(\varsigma)$ is a positive constant, 
we get
\begin{equation}
    \begin{split}
        \left\Vert \rho_{t}\right\Vert _{L^{\frac{6}{5}}(\mathcal{D})}  &  \leq\int_{0}^{t}\frac{c}{\left(
t-r\right)  ^{\gamma+\frac{s}{2}+\frac{1}{2}}}\left\Vert (-\Delta)^{-\gamma- \frac{1}{2}} \mathrm{div}\left(
u_{r}(-\Delta)^{\alpha_\theta }\xi_{r}\right)  \right\Vert _{L^{\frac{6}{5}}(\mathcal{D})}dr\\
&  \leq\left(  \int_{0}^{t}\left(  \frac{c}{\left(  t-r\right)  ^{\gamma
+\frac{s}{2}+\frac{1}{2}}}\right)  ^{\frac{p}{p-\lambda-\delta_u}%
}dr\right)  ^{\frac{p-\lambda-\delta_u}{p}}\left(  \int_{0}^{t}\left\Vert
(-\Delta)^{-\gamma- \frac{1}{2}} \mathrm{div}\left(  u_{r}(-\Delta)^{\alpha_\theta }\xi_{r}\right)  \right\Vert _{L^{\frac{6}{5}}(\mathcal{D})}%
^{\frac{p}{\lambda+\delta_u}}dr\right)  ^{\frac{\lambda+\delta_u}{p}}.
    \end{split}
    \label{eq: estimate rho}
\end{equation}
The second term is finite thanks to Lemma \ref{lemma: regolarita} (that we can apply since $\lambda  \leq 1-\delta_u $ and $\delta_\theta \in (0, \bar \delta_\theta]$), while the first term is finite thanks to \eqref{eq: auxiliary eq lambda delta}. Taking the supremum over $t \in [0,T]$ in \eqref{eq: estimate rho}, using Lemma \ref{lemma: regolarita} to bound the second term on the right-hand side of \eqref{eq: estimate rho}, and the boundedness of $(-\Delta)^{-s/2} \colon L^{6/5} \to W^{s, 6/5}$, we obtain \eqref{eq: estimate norm xi}.

\textbf{Step 3.} We now prove that \(t \mapsto \rho_t\) is continuous as a map \((0,T) \to L^{\frac{6}{5}}(\mathcal{D})\).  
For notational convenience, write
\[
\rho_t = \int_0^t (-\Delta)^{\gamma + \frac{s}{2} + \frac{1}{2}} e^{(t-\tau)\Delta} g_\tau \, d\tau, 
\qquad 
g_\tau := (-\Delta)^{-\gamma - \frac{1}{2}} \mathrm{div}\bigl(u_\tau (-\Delta)^{\alpha_\theta }\xi_\tau\bigr),
\]
so that \(g \in L^{\frac{p}{\lambda+\delta_u}}(0,T;L^{\frac{6}{5}}(\mathcal{D}))\) by Lemma~\ref{lemma: regolarita}.  
Fix \(r \in (0,T]\) and let \(t \in [r,T]\). Then
\begin{align*}
\rho_t - \rho_r
&= \int_0^t (-\Delta)^{\gamma + \frac{s}{2} + \frac{1}{2}} e^{(t-\tau)\Delta} g_\tau \, d\tau
   - \int_0^r (-\Delta)^{\gamma + \frac{s}{2} + \frac{1}{2}} e^{(r-\tau)\Delta} g_\tau \, d\tau \\
&= \int_r^t (-\Delta)^{\gamma + \frac{s}{2} + \frac{1}{2}} e^{(t-\tau)\Delta} g_\tau \, d\tau
   + \int_0^r \Bigl[(-\Delta)^{\gamma + \frac{s}{2} + \frac{1}{2}} e^{(t-\tau)\Delta}
                   - (-\Delta)^{\gamma + \frac{s}{2} + \frac{1}{2}} e^{(r-\tau)\Delta}\Bigr] g_\tau \, d\tau \\
&=: I_1(t,r) + I_2(t,r).
\end{align*}
Using the estimate for the fractional powers as in 
\eqref{eq: fractional power estimate for bound I2} and H\"older's inequality with 
\(q := \frac{p}{\lambda+\delta_u} > 1\), \(q' = \frac{q}{q-1}\), we obtain
\[
\|I_1(t,r)\|_{L^{\frac{6}{5}}}
\le C \Bigl( \int_r^t (t-\tau)^{-\bigl(\gamma + \frac{s}{2} + \frac{1}{2}\bigr) q'} \, d\tau \Bigr)^{1/q'}
      \Bigl( \int_r^t \|g_\tau\|_{L^{\frac{6}{5}}}^q \, d\tau \Bigr)^{1/q},
\]
for $C = C(\gamma, s, p, \lambda, \delta_u)>0$. The first factor $I_1$ is bounded uniformly in \(t \in [r,T]\) thanks to condition
\eqref{eq: auxiliary eq lambda delta}, while the second factor tends to \(0\) as \(t \downarrow r\) by the absolute continuity of the integral, since \(g \in L^q(0,T;L^{\frac{6}{5}}(\mathcal{D}))\). Hence \(\|I_1(t,r)\|_{L^{\frac{6}{5}}} \to 0\) as \(t \downarrow r\).

For the second term $I_2$, for every fixed \(\tau \in (0,r)\) the map
\[
\theta \mapsto (-\Delta)^{\gamma + \frac{s}{2} + \frac{1}{2}} e^{\theta \Delta} g_\tau,
\qquad \theta > 0,
\]
is continuous in \(L^{\frac{6}{5}}(\mathcal{D})\) by strong continuity of the heat semigroup; therefore the integrand in \(I_2(t,r)\) converges to \(0\) in \(L^{\frac{6}{5}}\) as \(t \downarrow r\) for a.e.\ \(\tau \in (0,r)\). Moreover, using again the estimate for the fractional power \eqref{eq: fractional power estimate for bound I2}, for \(t\) close to \(r\) we have
\[
\|(-\Delta)^{\gamma + \frac{s}{2} + \frac{1}{2}} e^{(t-\tau)\Delta} g_\tau
      - (-\Delta)^{\gamma + \frac{s}{2} + \frac{1}{2}} e^{(r-\tau)\Delta} g_\tau \|_{L^{\frac{6}{5}}}
\leq C (r-\tau)^{-\gamma - \frac{s}{2} - \frac{1}{2}} \|g_\tau\|_{L^{\frac{6}{5}}},
\]
and the right-hand side is integrable on \((0,r)\) by the same computation as in
\eqref{eq: estimate rho}. Hence, by the dominated convergence theorem,
\(\|I_2(t,r)\|_{L^{\frac{6}{5}}} \to 0\) as \(t \downarrow r\).  
We deduce that \(\|\rho_t - \rho_r\|_{L^{\frac{6}{5}}(\mathcal{D})} \to 0\) for every \(r \in (0,T]\), that is
\(\rho \in C(0,T; L^{\frac{6}{5}}(\mathcal{D}))\). Consequently,
\(\chi_t = (-\Delta)^{-\frac{s}{2}} \rho_t\) belongs to
\(C (0,T; W^{s,\frac{6}{5}}(\mathcal{D}))\).

\textbf{Step 4.} Existence of  $\gamma, \lambda $, given $s,p$. 
 Assume that, as before, $s \in (0,\frac{1}{2})$ and $p>1$ are given. We now check that we can always find $\gamma = \gamma(s)> \frac{1}{4}$ and $\lambda = \lambda (s,p)$ such that \eqref{eq: condition cor csi continuo} and \eqref{eq: secon condition csi continuo} are satisfied, and thus this corollary is non-trivial. Indeed, with the notation of Step $1$, from $a (\frac{1}{4}) = \frac{3}{4}+ \frac{s}{2}<1$, it follows that there exists $\gamma = \gamma (s) > \frac{1}{4}$ sufficiently close to $\frac{1}{4}$ such that $a(s,\gamma) <1$, that is \eqref{eq: condition cor csi continuo}. Hence, also $p (1-a) >0$ holds, and we may choose $\lambda = \lambda (s,p) \in (0,1)$ so small that $\lambda < p(1-a),$ that is \eqref{eq: secon condition csi continuo}. 
\end{proof}
\begin{remark}
From the proof of Corollary \ref{cor: csi continuo}, it follows that if 
    $s\in (0, \frac{1}{2}), p>1$, $\gamma > \frac{1}{4}$ and $\lambda\in (0,1)$ satisfy 
    \eqref{eq: condition cor csi continuo} and \eqref{eq: secon condition csi continuo}, then
    \[
    \bigl(\gamma + \tfrac{s}{2}+ \tfrac{1}{2}\bigr)\,
    \frac{1}{1- \frac{\lambda}{p}- \frac{\delta_u}{p} } <1
    \]
    for any $\delta_u < \bar \delta_u = \min \{p (1-a (s,\gamma))- \lambda, 1- \lambda \}$.
    \label{remark: condizione 3.9 vecchia dimostrazione}
\end{remark}

Lastly, we show the proof of Theorem \ref{thm: existence mild solution zeta}.
\begin{proof}[Proof of Theorem \ref{thm: existence mild solution zeta}]
For simplicity we set $\varepsilon=1$ and drop the superscript, writing $\zeta$ and $Z$ instead of $\zeta^\varepsilon$ and $Z^\varepsilon$. Let $s \in (0, \frac{1}{2}) $ and $p>4$ be given.

\textbf{Step 1.} Setup of parameters and auxiliary embeddings.
We work with the parameters $\gamma,\lambda,\bar\delta_\theta,\bar\delta_u$
fixed in the statement of the theorem. In particular,
\[
a(s,\gamma):=\frac{1}{2}+\gamma+\frac{s}{2}<1, \qquad
\lambda_{\max}:=p(1-a(s,\gamma)),
\]
and, by \eqref{eq:lmax-lmin}, we have $\lambda_{\max}>\frac12-s$. Moreover, $
\lambda_{\max}-(1-2s)>0 $
by the choice of $\varepsilon_\gamma$ in \eqref{eq:param-gamma} as shown in \eqref{eq: lambda max minus 12s positive}. We have set
\[
\lambda:= (\frac{1}{2}-s)+\varepsilon_\lambda,
\qquad 0<\varepsilon_\lambda<\min\{2s, \lambda_{\max}-(1-2s)\}.
\]
Note that
\[
\lambda<(\tfrac{1}{2}-s)+2s=\tfrac{1}{2}+s<1,
\qquad \lambda<(\tfrac{1}{2}-s)+\lambda_{\max}-(1-2s) =\lambda_{\max}-(\tfrac{1}{2}-s)<\lambda_{\max}.
\]
Hence $\lambda\in(\tfrac{1}{2}-s,\min\{1,\lambda_{\max}\})$. The thresholds are
\[
\bar{\delta}_\theta=\min\left \lbrace \lambda, 2\gamma-\frac{1}{2}\right \rbrace ,
\qquad \bar{\delta}_u=\min\left \lbrace \lambda_{\max}-\lambda,1-\lambda \right \rbrace.
\]
Finally, fix $0<\delta_\theta < \bar{\delta}_\theta$ and $0<\delta_u<\bar{\delta}_u$.

\emph{Choice of $q,q'$.} Set
\[
q:=\frac{6}{5-2s},\qquad \frac{1}{q}+\frac{1}{q'}=1.
\]
Then the Sobolev embedding
\begin{equation}\label{eq: first condition for q for thm 2.2}
W^{s,\frac{6}{5}}(\mathcal D)\hookrightarrow L^{q}(\mathcal D)
\end{equation}
holds, and a direct computation gives
\begin{equation}\label{eq: id_qprime}
1-\frac{3}{q'}=\frac12-s.
\end{equation}

\emph{Product embedding.} Since $\gamma>\frac14$, the Sobolev embedding yields
\begin{equation}
W^{2\gamma,6}(\mathcal D)\hookrightarrow L^\infty(\mathcal D).
\label{eq:W26 to Linfty}
\end{equation}
Moreover, again by the Sobolev embedding,
\begin{equation}
W^{\frac12+\lambda,2}(\mathcal D)\hookrightarrow L^{r_\lambda}(\mathcal D), \qquad r_\lambda:=\frac{3}{1-\lambda}.
\label{eq:W12l2 to Lr}
\end{equation}
Recalling that $
q'=\tfrac{6}{1+2s},$ it can be checked that $r_\lambda \geq q'$ is equivalent to $\lambda \geq \tfrac{1}{2}-s$. Hence, by \eqref{eq:W12l2 to Lr} and the embedding $L^{r_\lambda}(\mathcal D)\hookrightarrow L^{q'}(\mathcal D)$ on bounded $\mathcal D$,
\begin{equation}
W^{\frac12+\lambda,2}(\mathcal D)\hookrightarrow L^{q'}(\mathcal D).
\label{eq:W12l2 to Lq prime}
\end{equation}

Combining \eqref{eq:W26 to Linfty}, \eqref{eq:W12l2 to Lq prime} with H\"older's inequality yields the continuous multiplication map
\begin{equation}
W^{\frac12+\lambda,2}(\mathcal D)\times W^{2\gamma,6}(\mathcal D)
\hookrightarrow L^{q'}(\mathcal D).
\label{eq: coniditon multiplication embedding lambda thm 2.2}
\end{equation}
In particular, for every $\phi\in L^6(\mathcal D)$,
\begin{equation}
\bigl\|u_t\cdot\nabla(-\Delta)^{-\gamma-\frac12}\phi\bigr\|_{L^{q'}(\mathcal D)}
\le C\,\|u_t\|_{W^{\frac12+\lambda,2}(\mathcal D)}\,\|\phi\|_{L^6(\mathcal D)}.
\label{eq: mult_bound_for_phi}
\end{equation}

\emph{Time integrability condition.} Since $\delta_u<\bar\delta_u\le 1-\lambda$,
we have $\lambda\le 1-\delta_u$ and therefore Corollary~\ref{cor: interpolation regularity of u}
implies
\begin{equation}\label{eq: u_time_space_reg}
u\in L^{\frac{p}{\lambda+\delta_u}}\bigl(0,T;W^{\frac12+\lambda,2}(\mathcal D)\bigr).
\end{equation}
Moreover, $\delta_u<\bar\delta_u\le \lambda_{\max}-\lambda$ gives
$\lambda+\delta_u<\lambda_{\max}=p(1-a(s,\gamma))$, and hence, by
Remark~\ref{remark: condizione 3.9 vecchia dimostrazione},
\begin{equation}\label{eq: vecchia condizione 3.9 per thm 2.2}
\Bigl(\gamma+\frac{s}{2}+\frac{1}{2}\Bigr)\frac{1}{1-\frac{\lambda}{p}-\frac{\delta_u}{p}}<1.
\end{equation}

\textbf{Step 2.} Fixed point argument. Observe that, using the notation introduced in Corollary \ref{cor: csi continuo}, we can write $\zeta_t$ in the form
\[
\zeta_{t}=e^{t \Delta}\theta_0-(-\Delta)^{-\frac{s}{2}}  \int_{0}^{t}(-\Delta)^{\gamma+\frac{s}{2}+\frac{1}{2}
}e^{\left(  t-r \right)  \Delta} (-\Delta)^{-\gamma- \frac{1}{2}}\mathrm{div}\left(  u_{r}\zeta_{r}\right)  dr-\chi
_{t},
\] 
Let $X:= C\left(0,T_1; W^{s,\frac{6}{5}} (\mathcal{D})\right)$, where $T_1 \leq T$ will be chosen later, and for $\zeta \in X$ consider $\Psi\left(\zeta\right):= \tilde{\zeta}$, where $\tilde{\zeta}$ is defined as follows
\begin{equation}
\tilde{\zeta}_t:= e^{t \Delta}\theta_{0}-(-\Delta)^{-\frac{s}{2}}  \int_{0}^{t}(-\Delta)^{\gamma+\frac{s}{2}+\frac{1}{2}
}e^{\left(  t-r\right)  \Delta}(-\Delta)^{-\gamma - \frac{1}{2}}\mathrm{div}\left(  u_{r}\zeta_{r}\right)  dr-\chi
_{t}.
\label{eq: definition tilde zeta}
\end{equation}
We are going to apply a fixed point argument to obtain the thesis. First, since $\chi_t \in X$ thanks to Corollary \ref{cor: csi continuo}, to show that $\Psi(\zeta) \in X$ it is sufficient to prove that the linear application
\[
\Lambda\left(\zeta\right)_{t}:=\int_{0}^{t}(-\Delta) ^{\gamma+\frac{s}{2}+\frac{1}{2}}e^{\left(t-r\right)\Delta }(-\Delta) ^{-\gamma- \frac{1}{2}}\mathrm{div}\left(u_{r}\zeta_{r}\right)dr
\]
maps $X$ to $C(0,T_1;L^{\frac{6}{5}}(\mathcal{D}))$. Second, we will show that $\Psi\colon X \to X$ is a contraction. We divide the remaining part of the proof into three steps.

\textbf{Step 3.}
We want to show that, for any fixed $0 < \delta_u < \bar \delta_u (s,p)$ and any
$u$ as in the statement of the theorem, it holds
\begin{equation}
(-\Delta)^{-\gamma - \frac{1}{2}}\mathrm{div} \left(u_{t}\zeta_{t}\right)
\in L^{\frac{p}{\lambda+\delta_u}}\bigl(0,T_1;L^{\frac{6}{5}} (\mathcal{D})\bigr).
\label{eq: claim step 3 thm 22}
\end{equation}
Let $\phi \in L^{6}(\mathcal{D})$. As in the proof of Lemma~\ref{lemma: regolarita},
using integration by parts and the self-adjointness of $\Delta$, we obtain
\begin{align*}
\left\langle (-\Delta)^{-\gamma - \frac{1}{2}} \mathrm{div}\left(u_{t}\zeta_t\right),\phi\right\rangle
&= \left\langle \mathrm{div} (u_t \zeta_t), (-\Delta) ^{- \gamma -  \frac{1}{2}}  \phi \right\rangle \\
&= - \left\langle u_t \zeta_t , \nabla (-\Delta)^{- \gamma - \frac{1}{2}} \phi \right\rangle
 = - \left\langle \zeta_t , u_t \cdot \nabla (-\Delta)^{- \gamma - \frac{1}{2}} \phi \right\rangle .
\end{align*}
Let $q,q'$ be the conjugate exponents introduced in Step~1, namely
$q=\frac{6}{5-2s}$ and $1/q+1/q'=1$. By H\"older's inequality,
\[
\bigl|\left\langle \zeta_t , u_t \cdot \nabla (-\Delta)^{- \gamma - \frac{1}{2}} \phi \right\rangle\bigr|
\le \|\zeta_t\|_{L^{q}(\mathcal D)}\;
\|u_t \cdot \nabla (-\Delta)^{- \gamma - \frac{1}{2}} \phi\|_{L^{q'}(\mathcal D)}.
\]
Using the product estimate \eqref{eq: mult_bound_for_phi} from Step~1, we infer
\[
\|u_t \cdot \nabla (-\Delta)^{- \gamma - \frac{1}{2}} \phi\|_{L^{q'}(\mathcal D)}
\le C\,\|u_t\|_{W^{\frac12+\lambda,2}(\mathcal D)}\,\|\phi\|_{L^6(\mathcal D)},
\]
and therefore
\[
\bigl|\left\langle (-\Delta)^{-\gamma - \frac{1}{2}} \mathrm{div}\left(u_{t}\zeta_t\right),\phi\right\rangle\bigr|
\le C\,\|\zeta_t\|_{L^{q}(\mathcal D)}\;\|u_t\|_{W^{\frac12+\lambda,2}(\mathcal D)}\;\|\phi\|_{L^6(\mathcal D)}.
\]
Taking the supremum over $\|\phi\|_{L^6}\le 1$ yields
\begin{equation}\label{eq:step3_pointwise_bound}
\|(-\Delta)^{-\gamma - \frac{1}{2}} \mathrm{div}(u_t\zeta_t)\|_{L^{\frac{6}{5}}(\mathcal D)}
\le C\,\|\zeta_t\|_{L^{q}(\mathcal D)}\;\|u_t\|_{W^{\frac12+\lambda,2}(\mathcal D)}.
\end{equation}
Since $\zeta\in X=C(0,T_1;W^{s,\frac65}(\mathcal D))$, the Sobolev embedding
\eqref{eq: first condition for q for thm 2.2} implies
$\sup_{t\in[0,T_1]}\|\zeta_t\|_{L^{q}}<\infty$. Hence, raising
\eqref{eq:step3_pointwise_bound} to the power $\frac{p}{\lambda+\delta_u}$ and integrating,
we obtain
\begin{equation}
\int_0^{T_1}\|(-\Delta)^{-\gamma - \frac{1}{2}} \mathrm{div}(u_t\zeta_t)\|_{L^{\frac{6}{5}}}^{\frac{p}{\lambda+\delta_u}}\,dt
\le C\Bigl(\sup_{t\in[0,T_1]}\|\zeta_t\|_{L^q}\Bigr)^{\frac{p}{\lambda+\delta_u}}
\int_0^{T_1}\|u_t\|_{W^{\frac12+\lambda,2}}^{\frac{p}{\lambda+\delta_u}}\,dt.
\label{eq: bound norm u tau in Lp step one fixed point tau}
\end{equation}
The last integral is finite by \eqref{eq: u_time_space_reg}, i.e.
$u\in L^{\frac{p}{\lambda+\delta_u}}(0,T_1;W^{\frac12+\lambda,2}(\mathcal D))$.
This proves \eqref{eq: claim step 3 thm 22}

\textbf{Step 4.} Repeating the same arguments in Corollary \ref{cor: csi continuo}, we obtain $\Lambda\left(\zeta\right)_{t} \in C\left(0,T;L^{\frac{6}{5}} (\mathcal{D})\right)$, and thus $\Psi(\zeta) =\tilde{\zeta}\in X$.

\textbf{Step 5.} We conclude the proof by showing that $\Psi$ is a contraction in
$X=C(0,T_1;W^{s,\frac{6}{5}}(\mathcal D))$ for $T_1>0$ sufficiently small.
Let $\zeta_1,\zeta_2\in X$ and set $\zeta:=\zeta_1-\zeta_2$. From
\[
\Psi(\zeta_i)=e^{t\Delta}\theta_0-(-\Delta)^{-\frac{s}{2}}\Lambda(\zeta_i)_t-\chi_t,
\]
we deduce that, for every $t\in[0,T_1]$,
\[
\Psi(\zeta_1)(t)-\Psi(\zeta_2)(t)
=-(-\Delta)^{-\frac{s}{2}}\int_0^t (-\Delta)^{\gamma+\frac{s}{2}+\frac{1}{2}}
e^{(t-r)\Delta} (-\Delta)^{-\gamma-\frac{1}{2}}\mathrm{div}\bigl(u_r  \zeta_r\bigr)\,dr.
\]
Set
\[
\sigma:=\gamma+\frac{s}{2}+\frac{1}{2},\qquad
h(r):=(-\Delta)^{-\gamma-\frac{1}{2}}\mathrm{div}(u_r   \zeta_r),
\]
and recall that $q:=\frac{p}{\lambda+\delta_u}$ and $q'=\frac{q}{q-1}$. Using the smoothing estimate for the heat semigroup, we obtain for every $t\in[0,T_1]$
\begin{equation*}
    \begin{split}
        \|\Psi(\zeta_1)(t)-\Psi(\zeta_2)(t)\|_{W^{s,\frac{6}{5}}(\mathcal{D})}
&\leq C\int_0^t \|(-\Delta)^{\sigma}e^{(t-r)\Delta}h(r)\|_{L^{\frac{6}{5}}(\mathcal{D})}\,   dr\\
&\leq C \int_0^t (t-r)^{-\sigma} \|h(r)\|_{L^{\frac{6}{5}}(\mathcal{D})}\, dr.
    \end{split}
\end{equation*}
Applying H\"older's inequality in time with exponents $q,q'$, we get
\begin{align*}
\|\Psi(\zeta_1)(t)-\Psi(\zeta_2)(t)\|_{W^{s,\frac{6}{5}}}
&\leq C\left(\int_0^t (t-r)^{-\sigma q'}\,dr\right)^{\frac{1}{q'}}
\left(\int_0^t \|h(r)\|_{L^{\frac{6}{5}}}^{q}\,dr\right)^{\frac{1}{q}}\\
&\le C\left(\int_0^{T_1} \tau^{-\sigma q'}\,d\tau\right)^{\frac{1}{q'}}
\|h\|_{L^{q}(0,T_1;L^{\frac{6}{5}}(\mathcal D))}.
\end{align*}
The integral in $\tau$ is finite thanks to condition
\eqref{eq: vecchia condizione 3.9 per thm 2.2}, i.e.
$\sigma q'<1$. Hence,
\begin{equation}
\|\Psi(\zeta_1)-\Psi(\zeta_2)\|_{X}
\leq C (\int_0^{T_1} \tau^{-\sigma q'}\,d  \tau)^{\frac{1}{q'}}
\|h\|_{L^{q}(0,T_1;L^{\frac{6}{5}}( \mathcal{D}  ) )}.
\label{eq:Step5_kernel_bound}
\end{equation}
Finally, repeating the argument of Step3 in \eqref{eq:step3_pointwise_bound}-\eqref{eq: bound norm u tau in Lp step one fixed point tau} we obtain
\begin{equation*}
\|h\|_{L^{q}(0,T_1;L^{\frac{6}{5}}(\mathcal{D}))}
\leq C \|u\|_{L^{\frac{p}{\lambda+\delta_u}}\left(0,T_1;W^{\frac{1}{2}+\lambda,2}(\mathcal{D})\right)}
\|\zeta\|_{X}.
\end{equation*}
Substituting into \eqref{eq:Step5_kernel_bound} yields
\begin{equation}
\|\Psi(\zeta_1)-\Psi(\zeta_2)\|_{X} \leq K(T_1)
\|u\|_{L^{\frac{p}{\lambda+\delta_u}}(0,T_1;W^{\frac{1}{2}+\lambda,2}(\mathcal{D}))}
   \|\zeta_1-\zeta_2\|_{X},
\label{eq: estimate Psi}
\end{equation}
where
\[
K(T_1):=C\left(\int_0^{T_1}\tau^{-\sigma q'}\, d\tau \right)^{\frac{1}{q' }}.
\]
Since $u\in L^{\frac{p}{\lambda+\delta_u}}(0,T;W^{\frac{1}{2}+\lambda,2}(\mathcal D))$,
we have
$\|u\|_{L^{\frac{p}{\lambda+\delta_u}}(0,T_1;W^{\frac{1}{2}+\lambda,2})}\to 0$
as $T_1\downarrow 0$, and moreover $K(T_1)<\infty$ for every $T_1>0$ by $\sigma q'<1$. Therefore, choosing $T_1>0$ sufficiently small makes the right-hand
side of \eqref{eq: estimate Psi} strictly less than $\|\zeta_1-\zeta_2\|_X$, so $\Psi$
is a contraction on $X$. This yields a unique fixed point
$\zeta\in C(0,T_1;W^{s,\frac{6}{5}}(\mathcal {D
}))$.
If $T_1<T$, we may iterate the fixed point argument finitely many times to cover
$[0,T]$.
\end{proof}

\begin{remark}
    Note that Corollary \ref{cor: estimate tilde zeta} follows by applying to the definition of $\tilde{\zeta}_t$ in \eqref{eq: definition tilde zeta} the same arguments that were used to prove \eqref{eq: estimate norm xi} and \eqref{eq: bound norm u tau in Lp step one fixed point tau}.
\end{remark}

\section{Velocity problem}
\label{sec: velocity problem}

Let $\delta_u>0$ be a small regularity exponent for the velocity. In this section, we consider the velocity problem \eqref{eq: 3D NSEs velocity only}, assuming that the temperature is a given forcing term $ f^\theta \in L^p (0,T; H^{-\frac{1}{2}-\delta_u} (\mathcal{D}))$. By exploiting the weak-Stokes operator $A_w \colon \prescript{}{0}{H}^{\frac{3}{2}-\delta_u}_\sigma (\mathcal{D}) \to \prescript{}{0}{H^{- \frac{1}{2}-\delta_u}_\sigma  (\mathcal{D})}$ introduced in \eqref{eq: definition weak stokes operator} as 
\[
\langle  A_w u, v \rangle =  \langle \nabla u, \nabla v \rangle_{{H^{\frac{1}{2}-\delta_u }}, {H^{- \frac{1}{2}+ \delta_u} }},    
\quad (u,v) \in \prescript{}{0}{H}^{\frac{3}{2}-\delta_u}_\sigma (\mathcal{D}) 
\times \prescript{}{0}{H^{\frac{1}{2}+\delta_u}_\sigma(\mathcal{D})},
\]
then the velocity problem can be reformulated as a Cauchy problem in $\prescript{}{0}{H^{- \frac{1}{2}-\delta_u}_\sigma  (\mathcal{D})}$ as follows
\begin{equation}
\left \lbrace
    \begin{aligned}
        \partial_t u + A_w  u &= P(- u \cdot  \nabla u +f^\theta ),  & t \in (0,T), \\
        u|_{t = 0} &= u_0, &
    \end{aligned}
    \right.
    \label{eq: NSEs with force}
\end{equation}
where $P$ denotes the Helmholtz projection, and $u_0 \in V_{ p}^{\delta_u}$ denotes the initial condition, with
\[
V_{p}^{\delta_u} = \left( \prescript{}{0}H^{-\frac{1}{2}-\delta_u }_\sigma (\mathcal{D}), \prescript{}{0} H^{\frac{3}{2}-\delta_u}_\sigma  (\mathcal{D}) \right)_{1-\frac{1}{p}, p}.
\]
Note that here, and in the remainder of the section, we drop the dependence of $u$ on $\varepsilon$.

In Section \ref{subsec: stokes operator in weak asetting and max reg}, we investigate the properties of the weak Stokes operator $A_w$, in particular the $\mathcal{H}^\infty $-calculus and maximal regularity. Then, in Section \ref{subsec: global for small data} we apply these results to prove a global for small-data result for \eqref{eq: NSEs with force}, according to the mild notion of solution stated in Definition \ref{def: mild solution weak sense 3D NSEs}.

\subsection{Stokes operator in weak setting, $\mathcal{H}^\infty$-calculus and maximal regularity}
\label{subsec: stokes operator in weak asetting and max reg}

We start this section by recalling how the weak Stokes operator $A_w$ inherits the property of a bounded $\mathcal{H}^\infty$-calculus with $\mathcal{H}^\infty$-angle $\phi^\infty_{A_{w}} =0$ from the (strong) Stokes operator $A$. This is obtained by following the approach in \cite[Section 5]{Pruss2018}.

From this, it follows the maximal regularity property for the (weak) Stokes operator and the mixed derivative theorem, which are summarised in Proposition \ref{prop: max regularity stokes operator negative space} and Lemma \ref{lem: mixed derivative for our case}, respectively. Then, we conclude the section showing how it is possible to estimate the non-linear convective term $u \cdot \nabla u$ in $L^p(0,T; \prescript{}{0}{H}^{-\frac{1}{2}-\delta_u}_\sigma (\mathcal{D}))$ using the maximal regularity space $\mathbb{E}_{T,p}^{\delta_u}$, for any $p > \frac{2}{1-\delta_u}$. This is key to prove the global existence result for small data in Section \ref{subsec: global for small data}.

\begin{proposition}
    The weak Stokes operator 
    \begin{equation*}
        A_w \colon \prescript{}{0}{H}^{\frac{3}{2}-\delta_u}_\sigma(\mathcal{D}) 
\to \prescript{}{0}{H^{-\frac{1}{2}-\delta_u}_\sigma}(\mathcal{D})
    \end{equation*}
    admits a bounded $\mathcal{H}^\infty$-calculus on $\prescript{}{0}{}H^{-\frac{1}{2}-\delta_u}_\sigma (\mathcal{D})$ with angle $\phi^\infty_{A_w} = 0$.
    \label{prop: H infty calculus weak stokes operator}
\end{proposition}

\begin{proof}
Let $A_0 = A$ and $X_0 =L^2_\sigma(\mathcal{D})$. By \cite[Theorem V.1.5.1 and Theorem 1.5.4]{Amann1995}, it follows that the pair $(X_0, A_0)$ generates an interpolation-extrapolation scale $(X_\alpha, A_\alpha)$ with respect to the complex interpolation functor. Note that for $\alpha\in (0,1)$, $A_\alpha$ is the $X_\alpha$-realization of $A_0$ (the restriction of $A_0$ to $X_\alpha$) and 
\[
X_\alpha=D(A_0^\alpha).
\]
Let $X_0^\sharp:=(X_0)'$ and $A_0^\sharp:=(A_0)'$ with $D(A_0^\sharp)=:X_1^\sharp$. Then $(X_0^\sharp,A_0^\sharp)$ generates an interpolation-extrapolation scale $(X_\alpha^\sharp,A_\alpha^\sharp)$, the dual scale, and by \cite[Theorem V.1.5.12]{Amann1995}, it holds that
\[
(X_\alpha)'=X^\sharp_{-\alpha}\quad\text{and}\quad (A_\alpha)'=A^\sharp_{-\alpha}
\]
for $\alpha\in \mathbb{R}$.  Further, the domain of the fractional power of the Stokes operator can be characterised as
\[
X_{\alpha} = D((-\Delta)^\alpha)  \cap L^2_\sigma(\mathcal{D}), \quad \alpha \in (0,1),
\]
see \cite{Giga1985}. In the particular case  of
\[
\alpha = 1- \theta  = \frac{3}{4}- \frac{\delta_u}{2},
\]
then from the characterisation of the fractional powers of the Dirichlet Laplacian, we get
\[
X_{1- \theta } = \left \lbrace u \in H^{\frac{3}{2}-\delta_u} (\mathcal{D}) \ : \ u |_{\partial \mathcal{D}} = 0 \right \rbrace  \cap L^2_\sigma (\mathcal{D}),
\]
where $u|_{\partial \mathcal{D}}$ is understood in trace-sense. Hence, the operator $A_{-\theta}$ from the scale $(X_\alpha, A_\alpha)$, $\alpha \in \mathbb{R},$ is given by
\[
 A_{- \theta} \colon  X_{1-\theta } \to X_{- \theta },
\]
where by reflexivity $X_{-\theta} = \left(X_\theta^\sharp  \right)'$. Since $A_0^\sharp \in \mathcal{H}^\infty(X_0^\sharp)$, we have
\[
X^\sharp_\theta  = D((A_0^\sharp)^\theta) = [X_0^\sharp, X_1^\sharp  ]_\theta = \left \lbrace  u \in H^{\frac{1}{2}+\delta_u}(\mathcal{D}) \ : \ u|_{\partial \mathcal{D}} = 0 \right \rbrace  \cap L^2_\sigma (\mathcal{D}).
\]
Moreover, we have $A_{-\theta} = (A^\sharp_\theta)'$. Since $A_0^\sharp$ admits a bounded $\mathcal{H}^\infty$-calculus on $X_0^\sharp$ with angle $0$, by duality the operator $ A_{-\theta} \colon X_{1- \theta}  \to X_{-\theta}$ has a bounded $\mathcal{H}^\infty $-calculus with $\mathcal{H}^\infty$-angle $\phi^\infty_{A_{-\theta }} = 0.$

It remains to identify $A_{-\theta}$ with the weak Stokes operator $A_w$. Since $A_{-\theta}$ is the closure of $A_0$ in $X_{-\theta}$, it coincides with $A_0$ on the dense subspace $X_1$. For any $u \in X_1$ and $v \in X_1^\sharp$, we have via integration by parts
\[
\langle A_{-\theta} u, v\rangle =  ( A_0 u, v )_{L^2(\mathcal{D})}  = \int_{\mathcal{D}} \nabla u \cdot \nabla v \, dx.
\]
Since $X_1^\sharp$ is dense in $X_\theta^\sharp$, this identity extends to all $v \in X_\theta^\sharp$ by interpreting the integral as duality pairing. Note that for $u \in X_{1-\theta}$ and $v \in X_\theta^\sharp$, the gradients satisfy $\nabla u \in H^{\frac{1}{2}-\delta_u}(\mathcal{D})$ and $\nabla v \in H^{-\frac{1}{2}+\delta_u}(\mathcal{D})$. Thus, the identity extends continuously with respect to $u$ to the duality pairing
\[
\langle A_{-\theta} u, v\rangle  =  \langle \nabla u, \nabla v \rangle_{{H^{\frac{1}{2}-\delta_u } }, {H^{- \frac{1}{2}+ \delta_u} }}.
\]
Since $X_1$ is dense in $X_{1-\theta}$, this identity holds for all $u \in X_{1-\theta} = \prescript{}{0}{H}^{\frac{3}{2}-\delta_u}_\sigma (\mathcal{D})$ and $v \in X_\theta^\sharp = \prescript{}{0}{H}^{\frac{1}{2}+\delta_u}_\sigma (\mathcal{D})$. Therefore, the abstract operator $A_{-\theta}$ coincides with the weak Stokes operator $A_w$ defined in \eqref{eq: definition weak stokes operator}, and $A_w$ inherits the bounded $\mathcal{H}^\infty$-calculus.
\end{proof}

The first result that we recall is the maximal regularity of the weak Stokes operator, which is a consequence of the $\mathcal{H}^\infty$-calculus recalled above, see \cite[Section II.4]{Pruss2016}.
\begin{proposition}[Maximal regularity for the weak-Stokes operator]
Let $T>0$ be a given time, $ J = (0,T)$, $\delta_u \in (0,1)$. Consider the problem 
\begin{equation}
\left \lbrace
    \begin{aligned}
        \partial_t z+ A_w  z &= g, & \quad &t \in  (0,T), \\
        z|_{t = 0} &= z_0. & 
    \end{aligned}
    \right.
    \label{eq: stokes equation negative space}
\end{equation}
\begin{enumerate}
    \item[(i)] If $g \in L^p (J; \prescript{}{0}H^{-\frac{1}{2}-\delta_u}_\sigma (\mathcal{D}))$ and $z_0 \in V_{p}^{\delta_u}$, then there exists a unique solution
    \[
    z \in \mathbb{E}_{T,p}^{\delta_u} = W^{1,p}(J; \prescript{}{0}H^{-\frac{1}{2}-\delta_u}_\sigma(\mathcal{D})) \cap L^p(J; \prescript{}{0}H^{\frac{3}{2}-\delta_u}_\sigma(\mathcal{D}))
    \]
    of \eqref{eq: stokes equation negative space}.
    \item[(ii)] There exists $C = C(p,\delta_u, T) >0$ such that
    \begin{equation}
        \norm{z}_{\mathbb{E}_{T,p}^{\delta_u}} \leq C \left( \norm{g}_{L^p(J; H^{-\frac{1}{2}-\delta_u}(\mathcal{D}))} + \norm{z_0}_{V_{p}^{\delta_u}} \right).
    \end{equation}
\end{enumerate}
\label{prop: max regularity stokes operator negative space}
\end{proposition}
Second, we recall the mixed derivative theorem, see e.g. \cite[Corollary 4.5.10]{Pruss2016}.
\begin{lemma}[Mixed derivative]
    Let $T>0$ be a given time and $J = (0,T)$. Then
    \begin{equation}
    \mathbb{E}_{T,p}^{\delta_u} =W^{1,p}(J; \prescript{}{0}H^{-\frac{1}{2}-\delta_u}_\sigma (\mathcal{D})) \cap L^p(J; \prescript{}{0}H^{\frac{3}{2}-\delta_u}_\sigma(\mathcal{D})) \hookrightarrow H^{r, p} (J; D(A_w^{1-r})),
        \label{eq:mixed derivative theorem embedding}
    \end{equation}
    for any $r \in [0,1].$
\label{lem: mixed derivative for our case}
\end{lemma}

We conclude with an auxiliary result, which is based on Lemma \ref{lem: mixed derivative for our case}, for the estimate of the non-linear convective term of the NSEs \eqref{eq: NSEs with force}. 
\begin{lemma}
Let $T>0$ be a given time, $\delta_u \in (0, 1)$ and set  $J = (0,T)$. If $ p > \frac{2}{1-\delta_u}$, then there exists $C = C(p, \delta_u, T)>0$ such that
\[
\norm{ P (u  \cdot \nabla v) }_{L^p(J; \prescript{}{0}H^{-\frac{1}{2}-\delta_u}_{\sigma} (\mathcal{D}))} \leq C   \norm{u}_{\mathbb{E}_{T,p}^{\delta_u}}  \norm{v}_{\mathbb{E}_{T,p}^{\delta_u}}
\]
for any $u,v \in \mathbb{E}_{T,p}^{\delta_u}$, where $P: H^{- \frac{1}{2}- \delta_u} (\mathcal{D}) \to \prescript{}{0}{H^{- \frac{1}{2}- \delta_u}_{\sigma}(\mathcal{D})}$ is the extension of the Helmholtz projection. 
    \label{lemma: estimate convective term}
\end{lemma}
\begin{proof}
    Assume that $s_1,s_2\in \mathbb{R}$ are such that the following pointwise multiplication is bilinear and continuous
    \begin{equation}
    \begin{aligned}
            W^{s_1,2}(\mathcal{D}) \times W^{s_2,2}(\mathcal{D}) &\hookrightarrow W^{-\frac{1}{2}- \delta_u,2  }(\mathcal{D}) \\
            (f,g) &\mapsto f  g.
    \end{aligned}
        \label{eq: a lemma stima termine convettivo}
    \end{equation}
    Combining the previous multiplication embedding with the boundedness of $P$ on $H^{- \frac{1}{2}-\delta_u}(\mathcal{D})$ and the norm equivalence $H^s (\mathcal{D}) \cong W^{s,2}(\mathcal{D})$, there exists $C = C( s_1,s_2, \delta_u)>0$ such that
    \[
    \norm{P(u \cdot \nabla v)}_{\prescript{}{0}H^{- \frac{1}{2}-\delta_u}_\sigma (\mathcal{D})} \leq C \norm{u}_{W^{s_1,2} (\mathcal{D})}  \norm{\nabla v}_{W^{s_2,2} (\mathcal{D})} \leq C \norm{u}_{W^{s_1,2} (\mathcal{D})}  \norm{v}_{W^{1+s_2,2} (\mathcal{D})} . 
    \]
    Considering the $L^p$ norm in time and applying H\"older's inequality, we obtain
    \begin{equation*}
        \begin{split}
             \norm{P(u \cdot \nabla v)}_{L^p(J; \prescript{}{0}H^{- \frac{1}{2}-\delta_u }_\sigma (\mathcal{D}))} &\leq C \norm{u \cdot \nabla v}_{L^p(J; H^{- \frac{1}{2}-\delta_u } (\mathcal{D}))}\\
             &\leq C \norm{u}_{L^\infty (J; W^{s_1,2} (\mathcal{D}))}  \norm{v}_{L^p(J;W^{1+s_2,2}(\mathcal{D}))}.
        \end{split}
    \end{equation*}
    Assuming that
    \begin{equation}
            \mathbb{E}_{T,p}^{\delta_u} \hookrightarrow L^\infty (J; H^{s_1} (\mathcal{D}))
            \label{eq: b lemma stima termine convettivo}
    \end{equation}
    and
    \begin{equation}
        \mathbb{E}_{T,p}^{\delta_u} \hookrightarrow L^p (J; H^{1+s_2} (\mathcal{D})),
        \label{eq: c lemma stima termine convettivo}
    \end{equation}
    we conclude, from the norm equivalence $H^s(\mathcal{D}) \cong W^{s,2}(\mathcal{D})$, that
    \[
\norm{P(u \cdot  \nabla v)}_{L^p(J; \prescript{}{0}H^{-\frac{1}{2}-\delta_u}_\sigma (\mathcal{D}))} \leq C  \norm{u}_{\mathbb{E}_{T,p}^{\delta_u}} \norm{v}_{\mathbb{E}_{T,p}^{\delta_u}}.
\]
It remains to check that \eqref{eq: a lemma stima termine convettivo}, 
\eqref{eq: b lemma stima termine convettivo} and \eqref{eq: c lemma stima termine convettivo} hold.
First, using Lemma~\ref{lem: mixed derivative for our case}, it holds 
\[
\mathbb{E}_{T,p}^{\delta_u} \hookrightarrow H^{r, p}(J; D(A_w ^{1-r}  )) ,
\]
for any $r \in [0,1]$. Second, using the notation of the proof of Proposition \ref{prop: H infty calculus weak stokes operator}, recall that the pair $(X_0,A_0)$, where $A_0$ denotes the Stokes operator $A_0 = A$, and $X_0 =L^2_\sigma (\mathcal{D})$,  generates the interpolation-extrapolation scale $(X_\alpha, A_\alpha)$. Further, the weak Stokes operator is $A_w = A_{-\theta} : X_{1-\theta} \to X_{-\theta}$ and $D(A_w^{1-r}) =X_{-\theta+ (1-r)} = \prescript{}{0}H_\sigma ^{2(1-\theta -r)} (\mathcal{D}) $. Since $1- \theta = \frac{3}{4}- \frac{\delta_u}{2}$ (see the proof of Proposition \ref{prop: H infty calculus weak stokes operator}), we have $D(A_w^{1-r}) = \prescript{}{0} H_{\sigma}^{\frac{3}{2}- \delta_u - 2r} (\mathcal{D})$, and in conclusion
\[
\mathbb{E}_{T,p}^{\delta_u} \hookrightarrow H^{r, p}(J; D(A_w ^{1-r}  )) 
\hookrightarrow H^{r, p}\bigl(J; H^{\frac{3}{2}- \delta_u - 2r} (\mathcal{D})\bigr).
\]
Further, it can be checked that
\eqref{eq: b lemma stima termine convettivo} holds if
\begin{equation}
    \left \lbrace
    \begin{aligned}
         r &> \frac{1}{p}, \\
         \frac{3}{2}- \delta_u - 2r &\geq s_1,
    \end{aligned}
    \right.
    \label{eq: auxiliary system}
\end{equation}
is satisfied for some $r \in [0,1]$. On the other hand, 
\eqref{eq: c lemma stima termine convettivo} holds if $s_2 \leq \frac{1}{2}- \delta_u$. Lastly, choosing
\[
s_1 :=  \frac{3}{2}- \delta_u - 2r, \quad s_2:= \frac{1}{2}-\delta_u,
\]
the pointwise multiplication embedding \eqref{eq: a lemma stima termine convettivo} holds, thanks to \cite[Theorem 8.2]{Behzadan} if $\delta_u \in (0, \frac{1}{2})$ and \cite[Theorem 8.1]{Behzadan} if $\delta_u \in [\frac{1}{2}, 1)$, if $r < \frac{1-\delta_u}{2}$. Thus, we can find $r \in (0,1)$ such that \eqref{eq: a lemma stima termine convettivo} and \eqref{eq: auxiliary system} hold if
\[
\frac{1}{p}< \frac{1-\delta_u}{2},
\]
which is true thanks to our hypothesis $p> \frac{2}{1-\delta_u}$. This completes the proof.
\end{proof}

\subsection{Global for small data for the 3D Navier-Stokes problem}
\label{subsec: global for small data}
The auxiliary results presented in the previous section will be employed to demonstrate Theorem \ref{thm: global for small data NSEs with force}, i.e. the global existence result for small data for the 3D Navier-Stokes problem \eqref{eq: NSEs with force}, within the weak setting for the Stokes operator $A_w$. In particular, if we assume that the initial condition $u_0$ and the forcing term $f^\theta $ are sufficiently small, we show that the solution, in the sense of Definition \ref{def: mild solution weak sense 3D NSEs}, exists on the whole time interval $[0,T]$.

\begin{proof}[Proof of Theorem \ref{thm: global for small data NSEs with force}]
    Let $v_*$ denote the reference solution for \eqref{eq: NSEs with force}, i.e., the solution to the linear problem
    \begin{equation}
    \left \lbrace
    \begin{aligned}
         \partial_t v_* + A_w  v_* &= Pf^\theta , & \quad & t \in  (0,T),\\
         v_* |_{t = 0} & = u_0, 
    \end{aligned}
    \right.
        \label{eq: reference equation}
    \end{equation}
    and define the ball 
    \[
    \mathbb{B}\left( v_*, \frac{\eta}{2} \right) = \left \lbrace  v' \in \mathbb{E}_{T,p}^{\delta_u} \; : \; \norm{v' - v_*}_{\mathbb{E}_{T,p}^{\delta_u}} \leq \frac{\eta}{2} \right \rbrace \subset \mathbb{E}_{T,p}^{\delta_u}.
    \]
    Consider the map
    \begin{equation*}
        \begin{aligned}
            \Gamma \colon \mathbb{B}\left(v_*, \frac{\eta}{2}  \right ) \subset \mathbb{E}_{T,p}^{\delta_u}&\to  \mathbb{E}_{T,p}^{\delta_u}\\
            \Gamma(v') &= v,
        \end{aligned}
    \end{equation*}
    where $v$ is the unique solution of
    \begin{equation*}
\left \lbrace
    \begin{aligned}
         \partial_tv+ A_w  v &= P(- v' \cdot  \nabla v' +f^\theta), & t \in  (0,T),\\
        v|_{t = 0} &= u_0. & 
    \end{aligned}
    \right.
\end{equation*}
We will show that, there exists $\tilde{\eta } = \tilde\eta (p, \delta_u,T)>0$ such that for any $\eta \in (0, \tilde \eta )$, the following two conditions hold
\begin{enumerate}
    \item[(i)] $\norm{\Gamma (v')- v_*}_{\mathbb{E}_{T,p}^{\delta_u}} \leq \frac{\eta}{2} $ for all $v' \in \mathbb{B}( v_*, \frac{\eta}{2} )$,
    \item[(ii)] $\norm{\Gamma(v_1')- \Gamma(v_2')}_{\mathbb{E}_{T,p}^{\delta_u}} \leq d \cdot \norm{v_1' - v_2'}_{\mathbb{E}_{T,p}^{\delta_u}}$, where $d \in [0,1)$, for all $ v_1', v_2' \in \mathbb{B}(v_*, \frac{\eta}{2} ).$
\end{enumerate}

    (i). Let $v= \Gamma(v')$ and define $w:= v - v_*$, which satisfies
    \begin{equation}
    \left \lbrace
    \begin{aligned}
        \partial_tw +A_w  w &= P(- v' \cdot  \nabla v'), & \quad &t \in  (0,T),\\
         w|_{t = 0} & = 0.
    \end{aligned}
    \right.
    \end{equation}
    By the maximal regularity of the weak Stokes operator (see Proposition \ref{prop: max regularity stokes operator negative space}), Lemma \ref{lemma: estimate convective term} and Jensen's inequality, we obtain
    \begin{equation*}
        \begin{split}
            \norm{w}_{\mathbb{E}_{T,p}^{\delta_u}} \leq C  \norm{P(v' \cdot  \nabla v')}_{L^p(J; \prescript{}{0}H^{-\frac{1}{2}-\delta_u  }_{\sigma}(\mathcal{D}))} \leq C\norm{v'}_{\mathbb{E}_{T,p}^{\delta_u}}^2 \leq 2 C \left( \norm{v'-v_*}^2_{\mathbb{E}_{T,p}^{\delta_u}} + \norm{v_*}_{\mathbb{E}_{T,p}^{\delta_u}}^2 \right),
        \end{split}
    \end{equation*}
    for a positive constant $C = C(p,\delta_u, T)$ and $J = (0,T).$ Further, again by the maximal regularity of the weak Stokes operator, there exists a constant $M=M(p,\delta_u, T)$  such that
    \begin{equation}
    \begin{split}
    \norm{v_*}_{\mathbb{E}_{T,p}^{\delta_u}} \leq M \left( \norm{f^\theta}_{L^p(J;H^{-\frac{1}{2}-\delta_u }(\mathcal{D}))}+ \norm{u_0}_{V_{p}^{\delta_u}} \right)\leq  \frac{\eta}{2} .
    \end{split}
        \label{eq: maximal regularity v_*}
    \end{equation}
    Thus, we have
    \begin{equation}
        \begin{split}
            \norm{v- v_*}_{\mathbb{E}_{T,p}^{\delta_u}} = \norm{w}_{\mathbb{E}_{T,p}^{\delta_u}} &\leq 2C\left( \norm{v'-v_*}^2_{\mathbb{E}_{T,p}^{\delta_u}} + \norm{v_*}_{\mathbb{E}_{T,p}^{\delta_u}}^2 \right)  \leq 2C  \left(\frac{\eta^2}{4} + \frac{\eta^2}{4} \right)  = C \eta^2.
        \end{split}
        \label{eq: norm v - v_*}
    \end{equation}
    By choosing $\eta$ such that
    \[
    \eta^2 \leq \frac{\eta}{2C},
    \]
    we conclude that $\Gamma $ maps $\mathbb{B}\left(v_*, \frac{\eta}{2}\right)$ into itself.
 
    (ii). Let $v_i = \Gamma(v_i')$ for $i =1,2,$ and define $\tilde{w}:= v_1-v_2$, $\tilde{w}' = v_1' - v_2'$. Then, $\tilde{w}$ satisfies
    \begin{equation}
    \left \lbrace
    \begin{aligned}
          \partial_t \tilde w +A_w  \tilde{w} &=P( - v_1' \cdot  \nabla v_1' + v_2'  \cdot \nabla v_2'), & t \in   (0,T),\\
         \tilde{w}|_{t = 0} & = 0.
    \end{aligned}
    \right.
    \end{equation}
    By the maximal regularity of the Stokes operator, Lemma \ref{lemma: estimate convective term}, the definition of $v_i'$ and \eqref{eq: maximal regularity v_*}, we have
    \begin{equation*}
    \begin{split}
        \norm{\Gamma(v_1')- \Gamma(v_2')}_{\mathbb{E}_{T,p}^{\delta_u}} &\leq C \norm{ - \tilde{w}' \cdot \nabla v_1' - v_2' \cdot  \nabla \tilde{w}'}_{L^p(J; H^{- \frac{1}{2}-\delta_u}(\mathcal{D}))} \\
        & \leq C \norm{\tilde{w}'}_{\mathbb{E}_{T,p}^{\delta_u}} \left( \norm{v_1'}_{\mathbb{E}_{T,p}^{\delta_u}} + \norm{v_2'}_{\mathbb{E}_{T,p}^{\delta_u}}\right)\\
        &\leq C \norm{\tilde{w}'}_{\mathbb{E}_{T,p}^{\delta_u}} \left( \eta  + \eta \right)\\
         &= 2C  \norm{v_1' -v_2'}_{\mathbb{E}_{T,p}^{\delta_u}} \eta ,
    \end{split}
    \end{equation*}
    where $C= C( p, \delta_u,T)$ is a positive constant changing from line to line. Up to choosing $\eta$ sufficiently small, (ii) is verified.
    
        In conclusion, by the Banach fixed point theorem, there exists a unique fixed point $v = \Gamma(v) \in \mathbb{B}(v_*, \frac{\eta}{2}).$ Furthermore, by the triangle inequality, \eqref{eq: maximal regularity v_*}, and \eqref{eq: norm v - v_*}, we also obtain
        \begin{equation*}
            \norm{v}_{\mathbb{E}_{T,p}^{\delta_u}} \leq \norm{v-v_*}_{\mathbb{E}_{T,p}^{\delta_u}} + \norm{v_*}_{\mathbb{E}_{T,p}^{\delta_u}} \leq \frac{\eta}{2} + \frac{\eta}{2} = \eta.
        \end{equation*}
        This concludes the proof.
\end{proof}

\section{Well-posedness for the 3D-coupled temperature-velocity model with boundary noise}
\label{sec: well posedness coupled model}

We are now in position to prove the main result of our work.
\begin{proof}[Proof of Theorem \ref{thm: main theorem}]
Fix $T>0$, $s\in(0,\tfrac12)$ and $p>4$. Choose, as in \eqref{eq:param-gamma}, 
\[
\gamma:=\frac{1}{4}+\varepsilon_\gamma, \qquad 0<\varepsilon_\gamma<\min\left \lbrace \frac{1-2s}{4}, \frac{(1-2s)(p-4)}{4p}\right \rbrace ,
\]
and set
\[
a(s,\gamma):=\frac{1}{2}+\gamma+\frac{s}{2}, \qquad \lambda_{\max}:=p(1-a(s,\gamma)),
\qquad \lambda_{\min}:=\frac{1}{2}-s.
\]
Then $\lambda_{\max}>\lambda_{\min}$, thanks to the same computation as in \eqref{eq:lmax-lmin}. Hence the interval $(\lambda_{\min},\lambda_{\max})$ is non-empty.
Thus, we may choose $\lambda \in (\lambda_{\min},\lambda_{\max})$. In particular, as in \eqref{eq:param-lambda}, we pick
\[
\lambda:=\frac{1}{2}-s+\varepsilon_\lambda
\qquad 0 < \varepsilon_\lambda < \min \{ 2s, \lambda_{\max} - (1-2s) \}.
\]
Lastly, as in \eqref{eq:param-bar-deltas}, define the thresholds
\[
\bar{\delta}_\theta:=\min\left \lbrace \lambda, 2\gamma-\frac{1}{2}\right \rbrace , \qquad
\bar{\delta}_u:=\min \left \lbrace 1-\lambda, \lambda_{\max}-\lambda\right \rbrace .
\]
With this choice we have
\[
1-\lambda>\frac{1}{2}-s, \qquad
\lambda_{\max}-\lambda>\frac{1}{2}-s, \qquad\text{hence}\qquad
\bar{\delta}_u>\frac{1}{2}-s.
\]

In preparation for the application of Sobolev embeddings used below and Theorem \ref{thm: global for small data NSEs with force}, we now fix $
\delta_u \in (0,\bar\delta_u)$ and $
0<\delta_\theta < \bar\delta_\theta$
satisfying the compatibility conditions \eqref{eq: compatability conditions}, namely
\[
\delta_u \geq \max\left \lbrace \delta_\theta, \frac{1}{2}-s \right \rbrace ,
\qquad
\delta_u<1-\frac{2}{p}.
\]
These constraints are consistent: indeed $p>4$ implies $1-\frac{2}{p}>\frac{1}{2}$ and, by the choice of $\lambda$, we have $\bar\delta_u>\frac{1}{2}-s$, hence one can pick $\delta_u$ in the non-empty interval
$\bigl[\frac{1}{2}-s,\min\{\bar\delta_u,1-\frac{2}{p}\}\bigr)$ and then choose any $\delta_\theta\in(0,\min\{\bar\delta_\theta,\delta_u\})$.
Finally, note that $\delta_u<1-\frac{2}{p}$ is equivalent to
$p>\frac{2}{1-\delta_u}$, which is the integrability requirement needed to apply
maximal regularity in Theorem~\ref{thm: global for small data NSEs with force}.

Define
\[
    \alpha_\theta = \frac{1 }{4}+ \frac{\delta_\theta}{2},  \qquad
    \beta_\theta =\frac{1}{4}- \frac{\delta_\theta}{4 }.
\]
Let $\tilde\eta=\tilde \eta (p, \delta_u, T)>0 $ and $M =M(p, \delta_u,T)>0$ be the constants given by Theorem \ref{thm: global for small data NSEs with force}. Since $\delta_u  \geq\max\{ \delta_\theta, \frac{1}{2}-s    \}$ thanks to the compatibility conditions \eqref{eq: compatability conditions}, the embeddings
\[
    H^{-2\alpha _\theta }(\mathcal{D}) \hookrightarrow H^{- \frac{1}{2}-\delta_u}   (\mathcal{D}),   \quad W^{s, \frac{6}{5}}(\mathcal{D}) \hookrightarrow H^{s-1}(\mathcal{D}) \hookrightarrow H^{-\frac{1}{2}-\delta_u} (\mathcal{D}),
\]
    hold. Thus, there exists a constant $C_{emb} = C_{emb}(\delta_\theta, \delta_u, s)$ such that for all $g_1 \in H^{- 2\alpha_\theta} (\mathcal{D})$ and $g_2 \in W^{s, \frac{6}{5}} (\mathcal{D})$, we have
    \[
    \| g_1 \|_{H^{- \frac{1}{2}-\delta_u} (\mathcal{D})} \leq C_{emb} \| g_1 \|_{H^{-2\alpha_\theta} (\mathcal{D})}, \quad \| g_2 \|_{H^{-\frac{1}{2}- \delta_u}(\mathcal{D})} \leq C_{emb}
 \| g_2 \|_{W^{s, \frac{6}{5}}(\mathcal{D})}.
 \]
    We define $\widetilde{M}:= \max \{2,M, M T^{1/p}C_{emb} \}$. For any $\eta \in (0, \tilde{ \eta})$, define the stopping time $\tau^\varepsilon \colon \Omega \to [0,T]$ as
    \[
    \tau^\varepsilon := \inf \left \lbrace t \in [0,T] \ : \  \norm{Z^\varepsilon_t}_{H^{-2 \alpha_\theta}(\mathcal{D})}  > \frac{\eta}{8 \widetilde{M}}\right \rbrace,
    \]
    with the understanding that $\tau^\varepsilon = T$ if the previous set is empty. Note that
    \[
    \left \lbrace \omega \ : \ \tau^\varepsilon = T \right \rbrace = \left \lbrace  \omega \ : \ \sup_{0 \leq t \leq T}  \norm{Z^\varepsilon_t}_{H^{-2 \alpha_\theta }(\mathcal{D})}  \leq \frac{\eta}{8 \widetilde{M}}\right \rbrace .
    \]
    From Proposition \ref{prop: regularity for Z_t}, we deduce that
    \[
    \mathbb{P} \left(  \tau^\varepsilon = T\right) \geq 1- \frac{64 \widetilde{M}^2 \varepsilon}{\eta^2} \,  C(\delta_\theta, T) \,  \sum \limits_{k} \lambda_k^2 \norm{(-\Delta)^{\beta_\theta } D e_k }_2^2,
    \]
    where $C= C(\delta_\theta , T)$ is a positive constant. Consider
    \[
    \mathbb{B}_1 := \left \lbrace z \in C(0,\tau^\varepsilon ; W^{s,\frac{6}{5}}(\mathcal{D})) \ : \ \norm{z}_{C(0,\tau^\varepsilon ;W^{s, \frac{6}{5}}(\mathcal{D}))} \leq \frac{\eta }{8 \widetilde{M}} \right \rbrace,
    \]
    and the map
    \begin{equation*}
        \begin{split}
            \Gamma \colon \mathbb{B}_1 &\to C(0,\tau^\varepsilon ;W^{s, \frac{6}{5}}(\mathcal{D})) \\
            \zeta' &\mapsto \zeta^{\varepsilon }
        \end{split}
    \end{equation*}
    where
    \[
    \zeta_{t}^{\varepsilon  } := e^{t  \Delta} \theta_0 - \int_0^{t } e^{(t  -r)\Delta} \left( u^{\varepsilon}_r \cdot \nabla \zeta'_r \right)\, dr - \int_0^{t  } e^{(t  -r) \Delta  }\left( u_r^{\varepsilon } \cdot \nabla Z^\varepsilon_r \right)\, dr, \quad 0 \leq t \leq \tau^\varepsilon .\\
    \]
    Here $u^{\varepsilon} \in \mathbb{E}_{\tau^\varepsilon,p}^{\delta_u} $ is the unique solution of
    \begin{equation*}
        \left \lbrace
        \begin{aligned}
            \partial_t u^{\varepsilon}+ A_w u^{\varepsilon} & =P(- u^{\varepsilon} \cdot  \nabla u^{\varepsilon}   - \theta_t^{ \varepsilon,\prime}    e_3), \quad & t \in  (0, \tau ^\varepsilon),\\
            u^\varepsilon|_{t = 0}  & = u_0,
        \end{aligned}
        \right.
    \end{equation*}
    where
    \[
    \theta_t^{ \varepsilon  ,\prime} := Z^\varepsilon_{t \wedge \tau^\varepsilon } + \zeta'_t.
    \]
    Note that $u^{\varepsilon}$ is well defined thanks to Theorem \ref{thm: global for small data NSEs with force}, once we check the corresponding smallness condition on the forcing $-\theta^{\varepsilon,\prime}e_3$. First, by the definition of $\tau^\varepsilon$ we have
    \[
    \sup_{0\le t\le T}\bigl\|Z^\varepsilon_{t\wedge\tau^\varepsilon}\bigr\|_{H^{-2\alpha_{\theta}}(\mathcal D)}
      = \sup_{0\le t\le T}
        \bigl\|Z^\varepsilon_{t\wedge\tau^\varepsilon}\bigr\|_{H^{-\frac12-\delta_{\theta}}(\mathcal D)}
      \le \frac{\eta}{8\widetilde M}.
    \]
    Further, by definition of $\mathbb{B}_1$, we also have
    \[
    \sup_{0 \leq t \leq \tau^\varepsilon} \| \zeta_t' \|_{W^{s, \frac{6}{5}}(\mathcal{D})} \leq \frac{\eta}{8 \widetilde{M}}.
    \]
    Using the embedding constant $C_{emb}$ defined above, we estimate the forcing in $L^p (0, \tau^\varepsilon; H^{-\frac{1}{2} - \delta_u} (\mathcal{D}))$
    \begin{equation*}
        \begin{split}
            \| \theta^{\varepsilon,\prime} \| _{L^p(0,\tau^\varepsilon; H^{-\frac12-\delta_u}(\mathcal D))} &\leq  T^{1/p} \sup_{0 \leq t \leq \tau^\varepsilon} \left(  \| Z_t^\varepsilon \|_{H^{-\frac{1}{2}-\delta_u}(\mathcal{D})} + \|\zeta_t' \|_{H^{-\frac{1}{2}-\delta_u}(\mathcal{D})} \right) \\
            & \leq T^{1/p}C_{emb} \left( \sup_{0 \leq t \leq \tau^\varepsilon}  \| Z_t^\varepsilon \|_{H^{-2 \alpha_\theta }(\mathcal{D})} + \sup_{0 \leq t \leq \tau^\varepsilon} \|\zeta_t' \|_{W^{s, \frac{6}{5}}(\mathcal{D})} \right) \\
            &\leq T^{1/p} C_{emb} ( \frac{\eta}{8 \widetilde{M}}+ \frac{\eta }{ 8 \widetilde{M}} )\\
            &= T^{1/p} C_{emb} \frac{\eta}{4 \widetilde{M}}.
        \end{split}
    \end{equation*}
    Since we defined $\widetilde{M} \geq M T^{1/p} C_{emb},$ we deduce
    \[
    \| \theta^{\varepsilon,\prime} \| _{L^p(0,\tau^\varepsilon; H^{-\frac12-\delta_u}(\mathcal D))}  \leq \frac{\eta}{4 M}.
    \]
    This is exactly the smallness assumption required in Theorem \ref{thm: global for small data NSEs with force}, so $u^\varepsilon$ is well defined. 
    
    To apply the Banach fixed point theorem, we first check that $\Gamma$ maps $\mathbb{B}_1$ into itself. We can use \eqref{eq: estimate tilde tau} by Corollary \ref{cor: estimate tilde zeta} since its hypothesis, namely $\delta_\theta < \bar \delta_\theta $ and $\delta_u < \bar \delta_u$ are satisfied; further, we have the embedding $\mathbb{E}_{\tau^\varepsilon,p }^{\delta_u} \hookrightarrow L^{\frac{p}{\lambda + \delta_u }}(0,\tau ^\varepsilon;  W^{\frac{1}{2}+ \lambda,2}(\mathcal{D})), $ which holds for our fixed $\lambda$ thanks to the assumption $\delta_u<1-\lambda$. Thus, 
    \begin{equation*}
        \begin{split}
            \norm{\zeta^{\varepsilon }}_{C(0, \tau^\varepsilon ; W^{s,\frac{6}{5}}(\mathcal{D}))} &\leq \sup_{0 \leq t \leq \tau^\varepsilon}\norm{e^{t  \Delta }}_{W^{s,\frac{6}{5}}(\mathcal{D})} \norm{\theta_0}_{W^{s, \frac{6}{5}}(\mathcal{D})}\\
            &\quad + C \norm{u^{\varepsilon }}_{\mathbb{E}_{\tau^\varepsilon,p }^{\delta_u}} \left( \norm{\zeta '}_{C(0, \tau^\varepsilon ; W^{s, \frac{6}{5}}(\mathcal{D}))} + \norm{\xi^\varepsilon}_{C(0,\tau^\varepsilon ; L^2(\mathcal{D}))} \right) \\
            & \leq \frac{\eta}{16 \widetilde{M}} +  C \eta \left(\frac{\eta}{8 \widetilde{M}}+\frac{\eta}{8 \widetilde{M}}  \right) \\
            &= \frac{\eta}{16 \widetilde{M}} + \frac{C \eta^2}{4\widetilde{M}}.
        \end{split}
    \end{equation*}
    Thus, if $\eta $ is small enough, $\Gamma(\mathbb{B}_1) \subset \mathbb{B}_1$.
    
    Now we verify the contraction property. Consider $\zeta_i^{\varepsilon} = \Gamma (\zeta_i')$, $i = 1,2$. By the definition of $\Gamma$, for every $t\in[0,\tau^\varepsilon]$ we have
    \begin{equation*}
        \begin{split}
            \zeta_1^{\varepsilon}(t) - \zeta_2^{\varepsilon}(t)
            &= - \int_0^{t} e^{(t-r)\Delta}
                  \Big( u_{1,r}^{\varepsilon} \cdot \nabla \zeta_{1,r}' 
                       - u_{2,r}^{\varepsilon} \cdot \nabla \zeta_{2,r}' \Big)\,dr \\
            &\quad - \int_0^{t} e^{(t-r)\Delta}
                  \Big( u_{1,r}^{\varepsilon} - u_{2,r}^{\varepsilon} \Big)
                  \cdot \nabla Z_r^\varepsilon \,dr.
        \end{split}
    \end{equation*}
    We split the first integral by adding and subtracting the same quantity and then, using Corollary \ref{cor: csi continuo} and the bilinear estimate proved in \eqref{eq: bound norm u tau in Lp step one fixed point tau}, we obtain
    \begin{equation*}
        \begin{split}
            \norm{\zeta_{ 1}^{\varepsilon }- \zeta_{ 2}^{\varepsilon }}_{C(0,\tau^\varepsilon ; W^{s, \frac{6}{5}})} 
            &\leq C\biggl(
               \|u^{\varepsilon}_1\|_{\mathbb{E}_{\tau^\varepsilon,p }^{\delta_u}}
                 \|\zeta _1' - \zeta _2'\|_{C(0,\tau^\varepsilon ;W^{s,\frac{6}{5}})}\\
            &\quad + \|u^{\varepsilon}_1 - u^{\varepsilon}_2\|_{\mathbb{E}_{\tau^\varepsilon,p }^{\delta_u}}
                    \|\zeta _2'\|_{C(0,\tau^\varepsilon ;W^{s,\frac{6}{5}})} \\
            &\quad + \|u^{\varepsilon}_1 - u^{\varepsilon}_2\|_{\mathbb{E}_{\tau^\varepsilon,p }^{\delta_u}}
                    \|Z^\varepsilon\|_{C(0,\tau^\varepsilon;H^{-2\alpha_\theta})}
            \biggr).
        \end{split}
    \end{equation*}
    Let $w^{\varepsilon}:= u_1^{\varepsilon} - u_2^{\varepsilon}$. It satisfies 
    \begin{equation*}
        \left \lbrace
        \begin{aligned}
             \partial_tw^{\varepsilon} + A_w w^{\varepsilon} &=P( -  u_1^{\varepsilon}  \cdot \nabla u^{\varepsilon}_1 + u_2^{\varepsilon} \cdot  \nabla u^{\varepsilon}_2 - \left(\zeta_{t,1}' - \zeta_{t,2}' \right)    e_3), \quad &t \in  (0, \tau^\varepsilon ),\\
            w^\varepsilon|_{t = 0}  & = 0.
        \end{aligned}
        \right.
    \end{equation*}
    By the maximal regularity of the weak Stokes operator (Proposition \ref{prop: max regularity stokes operator negative space}) and Lemma \ref{lemma: estimate convective term}, there exists $C= C(p, \delta_u,T)$ such that
\begin{equation*}
    \begin{split}
        \norm{w^{\varepsilon }}_{\mathbb{E}_{\tau^\varepsilon,p}^{\delta_u}} &\leq C \left( \norm{w^{\varepsilon} \cdot  \nabla u_1^{\varepsilon}}_{L^p(0,\tau^\varepsilon;H^{-\frac{1}{2}-\delta_u}(\mathcal{D}))}+ \norm{u_2^{\varepsilon} \cdot  \nabla w^{\varepsilon}}_{L^p(0,\tau^\varepsilon;H^{-\frac{1}{2}-\delta_u}(\mathcal{D}))} \right. \\
        & \quad \left. +\norm{\zeta'_1- \zeta'_2}_{L^p(0,\tau^\varepsilon; H^{-\frac{1}{2}-\delta_u}(\mathcal{D}))} \right) \\
        &\leq C \left(\norm{w^{\varepsilon}}_{\mathbb{E}_{\tau^\varepsilon,p}^{\delta_u}} \norm{u_1^{\varepsilon}}_{\mathbb{E}_{\tau^\varepsilon,p}^{\delta_u}} +\norm{u_2^{\varepsilon}}_{\mathbb{E}_{\tau^\varepsilon,p}^{\delta_u}}  \norm{w^{\varepsilon}}_{\mathbb{E}_{\tau^\varepsilon,p}^{\delta_u}} +\norm{\zeta'_1- \zeta'_2}_{L^p(0,\tau^\varepsilon; H^{-\frac{1}{2}-\delta_u}(\mathcal{D}))}\right)\\
        & \leq C \left( 2 \eta \norm{w^{\varepsilon}}_{\mathbb{E}_{\tau^\varepsilon,p}^{\delta_u}}+ \norm{\zeta'_1- \zeta'_2}_{L^p(0,\tau^\varepsilon; H^{-\frac{1}{2}-\delta_u}(\mathcal{D}))}\right).
    \end{split}
\end{equation*}
Up to choosing $\eta $ sufficiently small, using the embedding $W^{s, \frac{6}{5}} (\mathcal{D}) \hookrightarrow  H^{s-1}(\mathcal{D}) \hookrightarrow H^{-\frac{1}{2}-\delta_u}(\mathcal{D})$, which holds for $\delta_u \geq \frac{1}{2}-s,$ we obtain
\[
 \norm{w^{\varepsilon }}_{\mathbb{E}_{\tau^\varepsilon,p}^{\delta_u}} \leq  C\norm{\zeta'_1- \zeta'_2}_{L^p(0,\tau^\varepsilon; H^{-\frac{1}{2}-\delta_u}(\mathcal{D}))} \leq C \norm{\zeta'_1- \zeta'_2}_{C(0,\tau^\varepsilon;W^{s, \frac{6}{5}}(\mathcal{D}))}.
\]
Up to renaming $C$, we get 
\begin{equation*}
    \begin{split}
        \norm{\zeta_{ 1}^{\varepsilon }- \zeta_{ 2}^{\varepsilon }}_{C(0,\tau^\varepsilon ; W^{s, \frac{6}{5}}(\mathcal{D}))} &\leq C \left( \eta  \norm{\zeta_1'- \zeta_2'}_{C(0,\tau^\varepsilon;W^{s, \frac{6}{5}}(\mathcal{D}))} + \norm{\zeta_1' - \zeta_2'}_{C(0,\tau^\varepsilon;W^{s, \frac{6}{5}}(\mathcal{D}))} \frac{\eta}{8 \widetilde{M}} \right) \\
        &  = d  \norm{\zeta_1' - \zeta_2'}_{C(0,\tau^\varepsilon;W^{s, \frac{6}{5}}(\mathcal{D}))}
    \end{split}
\end{equation*}
where $d:=  C \eta(1+ \frac{1}{8\widetilde{M}}).$ Thus, $d< 1$ if $\eta$ is small enough.

Lastly, regarding the adaptedness of the solution process, fix a time $t_0 \in [0,T]$. For any $\omega \in \Omega$, the contraction principle applied in the proof above, starting with a deterministic initial condition, generates an adapted sequence $(\tilde{u}^{\varepsilon, t_0}_n (\omega), \tilde{\theta}^{\varepsilon, t_0}_n(\omega))_n$ which converges to a limiting sequence $(\tilde{u}^{\varepsilon,  t_0} (\omega), \tilde{\theta}^{\varepsilon, t_0}(\omega))$. This latter stochastic process is adapted, being the limit of adapted processes. By uniqueness, the restriction to $[0,t_0]$ of the solution $(u^{\varepsilon} (\omega), \tilde{\theta}^{\varepsilon}(\omega))$ needs to coincide with the adapted solution $(\tilde{u}^{\varepsilon,  t_0}, \tilde{\theta}^{\varepsilon, t_0})$. This proves adaptedness. The progressive measurability of the solution can be obtained by a similar reasoning. The proof is thus complete.
\end{proof}

\section*{Acknowledgments}

We thank the two anonymous reviewers for their careful reading of the manuscript and for their valuable comments and suggestions, which helped improve the paper. G.D.S. acknowledges the support from the DFG project FOR~5528. M.L. is supported by the Italian national inter-university PhD course in Sustainable Development and Climate change. M.L. produced this work while attending the PhD programme in
PhD in Sustainable Development And Climate Change at the University School for Advanced Studies IUSS
Pavia, Cycle XXXVIII, with the support of a scholarship financed by the Ministerial Decree no. 351 of 9th
April 2022, based on the NRRP - funded by the European Union - NextGenerationEU - Mission 4 "Education
and Research", Component 1 "Enhancement of the offer of educational services: from nurseries to
universities” - Investment 3.4 “Advanced teaching and university skills”. The authors thank Franco Flandoli and Eliseo Luongo for the fruitful discussion on the topic.

\end{document}